\documentclass{amsart}

\usepackage{enumerate}
\usepackage{amsmath}
\usepackage{amssymb,latexsym}
\usepackage{amsthm}
\usepackage{color}
\usepackage{enumitem}
\usepackage{graphicx}
\usepackage{amscd}
\usepackage{hyperref}

\newcommand{\cC}{{\mathcal C}}

\newcommand{\cL}{{\mathcal L}}

\newcommand{\F}{{\mathbb F}}

\newcommand{\Fq}{{\F_q}}

\newcommand{\Fqn}{\F_{q^n}}

\newcommand{\Fqt}{{\F_{q^t}}}

\newcommand{\im}{\textnormal{im} }

\newtheorem{theorem}{Theorem}[section]
\newtheorem{lemma}[theorem]{Lemma}
\newtheorem{corollary}[theorem]{Corollary}
\newtheorem{proposition}[theorem]{Proposition}

\DeclareMathOperator{\PG}{{PG}}

\DeclareMathOperator{\Aut}{{Aut}}
\DeclareMathOperator{\GL}{{GL}}
\DeclareMathOperator{\PGL}{{PGL}}
\DeclareMathOperator{\PGaL}{P\Gamma L}
\DeclareMathOperator{\GaL}{\Gamma L}

\DeclareMathOperator{\rk}{rk}

\DeclareMathOperator{\End}{End}

\theoremstyle{definition}

\newtheorem{remark}[theorem]{Remark}
\definecolor{byzantium}{rgb}{0.44, 0.16, 0.39}
\def\zhou#1 {\fbox {\footnote {\ }}\ \footnotetext { From Yue: {\color{byzantium}#1}}}

\def\giuseppe#1 {\fbox {\footnote {\ }}\ \footnotetext { From Giuseppe: {\color{blue}#1}}}

\def\giovanni#1 {\fbox {\footnote {\ }}\ \footnotetext { From Giovanni: {\color{blue}#1}}}

\title[A  large family of maximum scattered linear sets and MRD codes]{A  large family of maximum scattered linear sets of $\PG(1,q^n)$ and their associated MRD codes}

\author[G. Longobardi]{G. Longobardi\textsuperscript{\,1}}
	\address{\textsuperscript{1}Dipartimento di Tecnica e Gestione dei Sistemi Industriali\\
	Universit\`a degli Studi di Padova\\
	Stradella S. Nicola, 3\\
	36100 Vicenza, Italy}
\email{giovanni.longobardi@unipd.it}
\author[G.\ Marino]{Giuseppe Marino\textsuperscript{\,2}}
\address{\textsuperscript{2}Dipartimento di Matematica e Applicazioni ``Renato Caccioppoli''\\
	Universit\`a degli Studi di Napoli ``Federico II''\\
	Via Vicinale Cupa Cintia, 80126 Napoli, Italy}
\email{giuseppe.marino@unina.it, rtrombet@unina.it}
\author[R.\ Trombetti]{Rocco Trombetti\textsuperscript{\,2}}
\author[Y.\ Zhou]{Yue Zhou\textsuperscript{\,3}}
\address{\textsuperscript{3}College of Liberal Arts and Sciences, National University of Defense Technology, 410073 Changsha, China}
\email{yue.zhou.ovgu@gmail.com}

\begin{document}
\maketitle
\begin{abstract}
	The concept of linear set in projective spaces over finite fields was introduced by Lunardon \cite{lunardon_normal_1999} and it plays central roles in the study of blocking sets, semifields, rank-metric codes and etc. A linear set with the largest possible cardinality and rank is called maximum scattered. Despite two decades of study, there are only a limited number of maximum scattered linear sets of a line $\PG(1,q^n)$. In this paper, we provide a large family of new maximum scattered linear sets over $\PG(1,q^n)$ for any even $n\geq 6$ and odd $q$. In particular, the relevant family contains at least 
	\[ \begin{cases}
	\left\lfloor\frac{q^t+1}{8rt}\right\rfloor,& \text{ if }t\not\equiv 2\pmod{4};\\[8pt]
	\left\lfloor\frac{q^t+1}{4rt(q^2+1)}\right\rfloor,& \text{ if }t\equiv 2\pmod{4},
	\end{cases} \]
	inequivalent members for given $q=p^r$ and $n=2t>8$, where $p=\mathrm{char}(\F_q)$. This is a great improvement of previous results: for given $q$ and $n>8$, the number of inequivalent maximum scattered linear sets of $\PG(1,q^n)$ in all classes known so far, is smaller than $q^2\phi(n)/2$, where $\phi$ denotes Euler's totient function.	Moreover, we show that there are a large number of new maximum rank-distance codes arising from the constructed linear sets.

\end{abstract}

\section{Introduction}
Let $V$ be a vector space over $\F_{q^n}$ of dimension $r$ and $\Omega=\PG(V,q^n)=\PG(r-1,q^n)$. A set of points $L_U$ of $\Omega$  is called an $\F_q$-\emph{linear set} of \emph{rank} $k$ if it consists of the points defined by the non-zero elements of an $\F_q$-subspace $U$ of $V$ of dimension $k$, that is,
\[ L_U=\left\{\langle \mathbf{u} \rangle_{\F_{q^n}} :   \mathbf{u}\in U\setminus \{ \mathbf{0} \} \right\}. \]
The term \emph{linear} was introduced by Lunardon \cite{lunardon_normal_1999} who considered a special kind of blocking sets. In the pasting two decades after this work, linear sets have been intensively investigated and applied to construct and characterize various objects in finite geometry, such as blocking sets, two-intersection sets, translation spreads of the Cayley generalized Hexagon, translation ovoids of polar spaces, semifields and rank-metric codes. We refer to \cite{bartoli_maximum_2018,lavrauw_field_reduction_2015,polverino_linear_2010,polverino_connections_2020,sheekey_new_2016} and the references therein.

The most interesting linear sets are those satisfying certain extremal properties. Firstly, it is clear that $|L_U|\leq \frac{q^k-1}{q-1}$. When the equality is achieved, $L$ is called \emph{scattered}.  A scattered linear set $L_U$ of $\Omega$ with largest possible rank $k$ is called a \emph{maximum scattered linear set}.  In \cite{blokhuis_scattered_2000}, it is proved that the largest possible rank is $k=rn/2$ if $r$ is even, and $(rn-n)/2\leq k\leq rn/2$ if $r$ is odd. In particular, when $r=2$, i.e.,  $L_U$ is a maximum scattered linear set over a projective line, its rank $k$ equals $n$.

For a given linear set $L_U$ of rank $n$ of a projective line, by a suitable collineation of $\PG(1,q^n)$, we may always assume that the point $\langle (0,1)\rangle_{\F_{q^n}}$ is not in $L_U$. This means
\[ U=U_f :=\{(x,f(x)): x\in \F_{q^n}  \}, \]
for a {\it $q$-polynomial} $f(x)$ over $\F_{q^n}$;i.e, an element of the set
\[{\mathcal L}_{n,q}[x]=\left\{\sum_{i=0}^{n-1} c_i x^{q^i}\colon c_i\in \F_{q^n} \right\}.\] Since polynomials in this set define $\F_q$-linear maps of $\F_{q^n}$ seen as $\F_q$-vector space, they are also known in the literature as \emph{linearized polynomials}.
Given a $q$-polynomial $f$, we use $L_f$ to denote the linear set defined by $U_f$. It is not difficult to show that $L_f$ is scattered if and only if for any $z,y\in \F_{q^n}^*$ the condition
\[\frac{f(z)}{z} = \frac{f(y)}{y}\]
implies that $z$ and $y$ are $\F_q$-linearly dependent. Hence, a $q$-polynomials satisfying this condition is usually called a \emph{scattered polynomial} (over $\F_{q^n}$). The condition for a $q$-polynomial $f(x)$ to be scattered can be rephrased by saying that if $f(\gamma x)=\gamma f(x)$, for $x$ and $\gamma \in \F_{q^n}$ with $x \neq 0$, then $\gamma \in \F_q$.

Scattered polynomials are also strongly related to a topic in network coding theory: the rank-distance codes. More precisely, a \textit{rank-distance code} (or RD code for short) $\cC$ is a subset of the set of $m \times n$ matrices 
$\F^{m \times n}_q$ over
$\Fq$ endowed with the rank distance 
$$d(A, B) = \rk (A - B)$$
for any $A, B  \in \F_q^{m \times n}$. The \textit{minimum distance} of an RD code $\cC$, $|\cC|\geq 2$, is defined as
\begin{equation*}
d(\cC) = \min_{\underset{M \not = N}{M, N \in \cC}} d(M, N) \,.
\end{equation*}
A rank-distance code of $\F_q^{m \times n }$ with minimum distance $d$ has \emph{parameters}
$(m, n, q; d)$. 
If  $\cC$ is an $\F_q$-linear subspace of $\F_q^{m \times n}$, then $\cC$ is called $\F_q$-\textit{linear} 
RD code and its \emph{dimension} $\dim_{\F_q} \cC$ is defined to be the dimension of $\cC$ as 
a subspace over $\F_q$.
The \textit{Singleton-like bound}  \cite{Delsarte} for an $(m,n,q;d)$ RD-code $\cC$ is
$$|\cC| \leq q^{\max\{m,n\}(\min\{m,n\}-d+1)}.$$
If $\cC$ attains this size, then $\cC$ is a called  \textit{Maximum Rank-Distance code}, MRD \textit{code} for short.
In this paper we will consider only the case in which the codewords are square matrices, i.e. $m=n$.
Note that if $n = d$, then an MRD code $\cC$ consists of $q^n$ invertible endomorphisms of $\Fqn$; such $\cC$  is called \textit{spread set} of $\End_{\Fq}(\Fqn )$. 
In particular if $\cC$ 
is $\Fq$-linear, it is called a \textit{semifield spread set} of $\End_{\Fq}(\Fqn )$, see \cite{LavPol}.
Two $\Fq$-linear codes $\cC$ and $\cC'$ are called \textit{equivalent} if
there exist $A,B \in \GL(n, q)$ and a field automorphism $\sigma$ of $\Fq$ such that 
$$\cC' = \{AC^\sigma B \colon C \in \cC\}.$$
The aforementioned link lies in the fact that rank-distance codes can be described by means of $q$-polynomials over $\F_{q^n}$, considered modulo $x^{q^n}-x$. 
After fixing an ordered $\F_q$-basis $\{b_1,b_2,\ldots,b_n\}$ for $\F_{q^n}$ it is possible to give a bijection $\Phi$ which associates for each matrix $M\in \F_q^{n\times n}$ a unique $q$-polynomial $f_M\in {\mathcal L}_{n,q}[x]$. More precisely, put ${\bf b}=(b_1,b_2,\ldots,b_n)\in \F_{q^n}^n$, then $\Phi(M)=f_M$ where for each ${\mathbf u}=(u_1,u_2,\ldots,u_n)\in \F_q^n$ we have $f_M({\mathbf b}\cdot {\mathbf u})={\mathbf b}\cdot {\mathbf u}  M$.

Given a scattered polynomial $f$,  the set of $q$-polynomials 
\[\cC_f=\{ ax+bf(x): a,b\in \F_{q^n} \}\]
defines a linear MRD code of minimum distance $n-1$ over $\F_q$. For recent surveys on MRD codes and their relations with linear sets, we refer to \cite{polverino_connections_2020,sheekey_MRD_survey_arxiv}.

Up to now, there are only three families of maximum scattered linear sets in $\PG(1,q^n)$ for infinitely many $n$. We list the corresponding scattered polynomials over $\F_{q^n}$ below:
\begin{enumerate}[label=(\roman*)]
	\item\label{family:PR} $x^{q^s}$, where $1 \leq s \leq n-1$ and $\gcd(s,n)=1$; see \cite{blokhuis_scattered_2000}.
	\item\label{family:LP} $\delta x^{q^s}+x^{q^{n-s}}$, where $n \geq 4$,\, $N_{q^n/q}(\delta)\notin \{0,1\}$, $\gcd(s,n)=1$ and $N_{q^n/q}\,:\,x \in \F_{q^n} \mapsto x^{\frac{q^n-1}{q-1}} \in \Fq,$ is the norm function of $\F_{q^n}$ over $\F_q$; see \cite{lunardon_blocking_2001,sheekey_new_2016}.
	\item\label{family:LZ} $\psi^{(k)}(x)$,  where $\psi(x)= \frac{1}{2}\left(x^q+x^{q^{t-1}} -x^{q^{t+1}}+x^{q^{2t-1}}\right)$, $q$ odd, $n=2t$ and 
	\begin{itemize}
		\item [-] $t$ is even and $\gcd(k,t)=1$, or
		\item [-] $t$ is odd, $\gcd(k,2t)=1$, and $q\equiv 1\pmod{4}$; see \cite{longobardi_familyOfLinearMRD_2021}. 
	\end{itemize}
\end{enumerate}
For $n\in\{6,8\}$,  there are other  families of  scattered polynomials over $\F_{q^n}$; see \cite{bartoli_new_maximum_2020,csajbok_newMRD_2018,csajbok_new_maximum_2018,marino_mrd-codes_2020,zanella_vertex_2020}. According to the asymptotic classification results of them obtained in \cite{bartoli_classification_2021,bartoli_scattered_2018},  maximum scattered linear sets in $\PG(1,q^n)$ seem rare.

Two linear sets $L_U$ and $L_{U'}$ of $\PG(1,q^n)$ are said to be $\PGaL$-\textit{equivalent} (or  \textit{projectively equivalent}) if there exists $\varphi\in \PGaL(2,q^n)$ such that $L_U^\varphi=L_{U'}$. For two given $q$-polynomials,  it is well-known that $\cC_f$ is equivalent to $\cC_g$ if and only if $U_f$ and $U_g$ are  on the same $\GaL(2,q^n)$-orbit, which further implies that $L_f$ and $L_g$ are $\PGaL$-equivalent. However, the converse statement is not true in general. For instance,  if $U_f=\left\{(x, x^q) :  x\in \F_{q^n} \right\}$ and $U_g=\left\{(x, x^{q^s}) :  x\in \F_{q^n} \right\}$ with $s\neq 1$ and $\gcd(s,n)=1$, then $U_f$ and $U_g$ are not $\GaL(2,q^n)$-equivalent, but obviously $L_f=\left\{\langle (1, x^{q-1})\rangle_{\F_{q^n}} :  x\in \F_{q^n}^* \right\}=L_g$. For more results on the equivalence problems, we refer to \cite{csajbok_classes_2018,csajbok_equivalence_2016}.

\begin{table}
	\centering
	\begin{tabular}{cccc}
		\hline 
		No. &Families & \# inequivalent $\cC_f$  & \# inequivalent  $L_f$ \\
		\hline 
		\ref{family:PR}& Pseudo-regulus & $\phi(n)/2$ & $1$ \\[4pt]
		\ref{family:LP} &Lunardon-Polverino & $\leq\begin{cases}
		{\phi(n)}\frac{q-2}{{2}}, & 2\nmid n,\\[4pt]
		{\phi(n)}\frac{(q+1)(q-2)}{{4}}, & 2\mid n,
		\end{cases}$ &   $\leq\begin{cases}
		{\phi(n)}\frac{q-2}{{2}}, & 2\nmid n,\\[4pt]
		{\phi(n)}\frac{(q+1)(q-2)}{{4}}, & 2\mid n,
		\end{cases}$   \\ [4pt]
		\ref{family:LZ}&Longobardi-Zanella & $\phi(n)/2$ & $\leq \phi(n)/2$ \\
		\hline
	\end{tabular} 
	
\medskip
	\caption{Numbers of inequivalent $\cC_f$ and $\PGaL(2,q^n)$-inequivalent $L_f$, where $f$ is a scattered polynomial in \ref{family:PR}, \ref{family:LP} or \ref{family:LZ}, $q=p^r$ with $p=\mathrm{char}(\F_q)$ and $\phi$ denotes Euler's totient function.}\label{table:numbers}
\end{table}

In Table \ref{table:numbers}, we list the numbers of inequivalent $\cC_f$ and $\PGaL(2,q^n)$-inequivalent $L_f$, for a scattered polynomial $f$ in each one of the three known families. The proof of what is stated in Table \ref{table:numbers} will be provided in Section \ref{sec:pre}.

In this paper, we present a new family of maximum scattered linear sets in $\PG(1,q^n)$ where $q=p^r$, $p$ is an odd prime  and $n=2t\geq 6$; see Theorem \ref{t:main}. In particular, when $t>4$, this new family provides at least 
\[ \begin{cases}
\left\lfloor\frac{q^t+1}{4rt}\right\rfloor,& \text{ if }t\not\equiv 2\pmod{4};\\[8pt]
\left\lfloor\frac{q^t+1}{2rt(q^2+1)}\right\rfloor,& \text{ if }t\equiv 2\pmod{4}
\end{cases} \]
inequivalent number of $\F_{q^n}$-MRD codes (Corollary \ref{coro:BIG}) and  at least 
\[ \begin{cases}
\left\lfloor\frac{q^t+1}{8rt}\right\rfloor,& \text{ if }t\not\equiv 2\pmod{4};\\[8pt]
\left\lfloor\frac{q^t+1}{4rt(q^2+1)}\right\rfloor,& \text{ if }t\equiv 2\pmod{4}
\end{cases} \]
$\PGaL(2,q^n)$-inequivalent maximum linear sets (Theorem \ref{t:PGaL-equivalence}). Therefore, the number of maximum scattered linear sets in $\PG(1,q^n)$ (and hence of $\F_{q^n}$-MRD codes) grows exponentially with respect to $n$.

The remaining part of this paper is organized as follows. In Section \ref{sec:pre}, we introduce more results on the equivalence of maximum scattered linear sets in $\PG(1,q^n)$ and the associated MRD codes, and explain Table \ref{table:numbers} in details. In Section \ref{sec:construction}, we 
exhibit a family of scattered polynomials $f$ over $\F_{q^n}$ and provide its proof. The equivalence between the  MRD codes $\cC_f$ associated to the members of this family are completely determined in Section \ref{sec:MRD}, in which we also study the MRD codes derived from the adjoint maps of our scattered polynomials. Based on these results, the $\PGaL$-equivalence of the associated maximum linear sets are investigated in Section \ref{sec:PGaL_equivalence}. 

\section{Preliminaries}\label{sec:pre}
\subsection{Equivalence of MRD codes and linear sets}

Recall that, given two scattered polynomials $f$ and $g$ over $\F_{q^n}$, the corresponding MRD codes $\cC_f$ and $\cC_g$ are equivalent  if and only if there exist $L_1$, $L_2\in {\mathcal L}_{n,q}[x]$ permuting $\F_{q^n}$ and $\rho\in \Aut(\F_{q^n})$ such that
	\[ L_1\circ \varphi^\rho \circ L_2 \in \cC_g \text{ for all }\varphi\in \cC_f,\]
	where $\circ$ stands for the composition of maps and $\varphi^\rho(x)= \sum a_i^\rho x^{q^i}$ for $\varphi(x)=\sum a_i x^{q^i}$.

The following result concerning the equivalence of MRD codes associated with scattered polynomials  is proved in \cite{sheekey_new_2016}.
\begin{theorem}\label{t:code_equiv=space_equiv}
	Let $f$ and $g$ be two scattered polynomials over $\F_{q^n}$, respectively. The MRD-codes $\cC_f$ and $\cC_g$ are equivalent if and only if $U_f$ and $U_g$ are $\GaL(2,q^n)$-equivalent.
\end{theorem}
For this paper,  we only need the necessary and sufficient conditions in Theorem \ref{t:code_equiv=space_equiv} to interpret the equivalence problem in Section \ref{sec:MRD}. However, in general, the equivalence problem for MRD codes could be more complicated. We refer to the surveys \cite{sheekey_MRD_survey_arxiv,polverino_connections_2020} for its precise definition and related results. See  \cite{couvreur_hardness_arxiv} for the hardness of testing the equivalence between rank-distance codes.
The \textit{left idealizer}  and the \textit{right idealizer} of any given rank-distance code $\cC$ are invariant under equivalence. These two concepts were introduced in \cite{liebhold_automorphism_2016}, and in \cite{lunardon_kernels_2017} in different names. If a rank-distance code $\cC$ is given as a subset of $\cL_{n,q}[x]$, then its left idealizer and right idealizer are defined as
\[I_L(\cC) =\{\varphi\in \cL_{n,q}[x]: \varphi\circ f\in \cC \text{ for all }f\in \cC \}, \]
and
\[I_R(\cC) =\{\varphi\in \cL_{n,q}[x]: f\circ \varphi \in \cC \text{ for all }f\in \cC \}, \]
respectively. When $\cC$ is an MRD-code, it is well known that all nonzero elements in $I_L(\cC)$ and $I_R(\cC)$ are invertible and each of the idealizers must be a subfield of $\F_{q^n}$. 
In particular, if $\cC$ is an $\F_{q^n}$-subspace of $\cL_{n,q}[x]$, then $I_L(\cC)$ is isomorphic to $\F_{q^n}$ and $\cC$ is said to be an \emph{$\F_{q^n}$-MRD code}.
Obviously, for every MRD code $\cC_f$ associated with a scattered polynomial $f$ over $\F_{q^n}$, its left idealizer $I_L(\cC_f)$  is isomorphic to $\F_{q^n}$.

For a $q$-polynomial $f(x)=\sum_{i=0}^{n-1}{a_i}x^{q^i}$ over $\F_{q^n}$, the \emph{adjoint map} $\hat{f}$ of it with respect to the bilinear form $\langle x,y \rangle=\mathrm{Tr}_{q^n/q} (xy)$ is 
\[ \hat{f}(x) = a_0 x+\sum_{n=1}^{n-1} a_i^{q^{n-i}} x^{q^{n-i}} . \]
For a given scattered polynomial $f$ over $\F_{q^n}$, its adjoint $\hat f$ is a scattered polynomial over $\F_{q^n}$ as well, and $U_f$ and $U_{\hat{f}}$ (and hence $\cC_f$ and $\cC_{\hat f}$) are not necessarily equivalent. However, they define exactly the same linear set of $\PG(1,q^n)$; see \cite{bartoli_maximum_2018,csajbok_classes_2018}.

To investigate the $\PGaL$-equivalence among linear sets of a line, we need the following result proved in \cite{csajbok_classes_2018}.
\begin{lemma}\label{le:jcta2018}
	Let $f(x)=\sum_{i=0}^{n-1} \alpha_i x^{q^i}$ and $g(x)=\sum_{i=0}^{n-1} \beta_i x^{q^i}$ be two $q$-polynomials over $\F_{q^n}$ such that $L_f=L_g$. Then $\alpha_0=\beta_0$, and 
	\[ \alpha_k \alpha_{n-k} ^{q^k} = \beta_k \beta_{n-k}^{q^k} \]
	for $k=1,2,\cdots, n-1$, and 
	\[ \alpha_1 \alpha_{k-1}^q\alpha_{n-k}^{q^k}  + \alpha_k \alpha_{n-1}^{q} \alpha_{n-k+1}^{q^k}=   \beta_1 \beta_{k-1}^q\beta_{n-k}^{q^k}  + \beta_k \beta_{n-1}^{q} \beta_{n-k+1}^{q^k},\]
	for $k=2,3,\cdots, n-1$.
\end{lemma}

\subsection{Details of Table \ref{table:numbers}}
In the following  we provide details on the estimates stated in  Table \ref{table:numbers}. The number of inequivalent MRD codes defined in (i) comes from the well-known results on the equivalence of Delsarte-Gabidulin codes and their generalizations; see \cite[Theorem 4.4]{lunardon_generalized_2018} for instance. Moreover all monomials  $f(x)=x^{q^s}$ with $\gcd(s,n)=1$, determine the same linear set in $\PG(1,q^n)$; in fact,  
\[L_f:=\left\{\langle (1, x^{q^{s}-1}) \rangle_{\Fqn} : x\in \F_{q^n}^*  \right\};\] the so-called linear set of pseudo-regulus type.

For the equivalence of MRD codes defined by \ref{family:LP}, we need the following result which is a special case of \cite[Theorem 4.4]{lunardon_generalized_2018}.
\begin{proposition}\label{le:LTZ_2018}
	For $\theta,\eta \in \F_{q^n}$ such that $N_{q^n/q}(\theta), N_{q^n/q}(\eta)\notin \{0,1\}$, with $1\leq s,t \leq \frac{n-1}{2}$ satisfying $\gcd(s,n)=1$, let $f(x) = \eta x^{q^s}+x^{q^{n-s}}$ and $g(x)=\theta x^{q^t}+x^{q^{n-t}}$. Then $\cC_f$ and $\cC_g$ are equivalent if and only if $s=t$ and 
	$$\theta=\eta^{\tau} z^{q^{2s}-1}$$
	for some $\tau\in \Aut(\F_{q^n})$ and $z\in \F_{q^n}^*$.
\end{proposition}

By Hilbert's Theorem 90, if $m | n$, $ \{x\in \F^*_{q^n}: N_{q^n/q^{m}}(x)=1  \} = \{ y^{q^{m}-1}: y\in \F^*_{q^n} \}$. As
\[\gcd(q^{2s}-1,q^n-1) = q^{\gcd(2s,n)}-1=
\begin{cases}
	q^2-1, &2\mid n,\\
	q-1, & 2\nmid n,
\end{cases}
\]
if $N_{q^n/q^{\gcd(2s,n)}}(\theta)= N_{q^n/q^{\gcd(2s,n)}}(\eta)$, then we can always find $z\in \F_{q^n}^*$ such that $\theta = \eta z^{q^{2s}-1}$.  Hence, under the maps $\eta \mapsto \eta z^{q^{2s}-1}$ for $z\in \F_{q^n}^*$,  the elements in $\F_{q^n}^*$ are partitioned into $q^{\gcd(2s,n)}-1$ orbits. Moreover, $\theta_1$ and $ \theta_2$ are in the same orbit under the maps $\eta\mapsto \eta^\tau z^{q^{2s}-1}$  for $z\in \F_{q^n}^*$ and $\tau\in \Aut(\F_{q^n})$ if and only if $N_{q^n/q^{\gcd(2s,n)}}(\theta_1)= \left(N_{q^n/q^{\gcd(2s,n)}}(\theta_2)\right)^{\tau'}$ for some $\tau'\in \Aut(\F_{q^{\gcd(2s,n)}})$.
Note that for most choices of $\theta$ satisfying $N_{q^n/q}(\theta)\notin \{0,1\}$, $N_{q^n/q^{\gcd(s,n)}}(\theta)$ is not in any proper subfield of $\F_{q^{\gcd(s,n)}}$. 
Therefore, by Proposition \ref{le:LTZ_2018} the number of inequivalent MRD codes from family \ref{family:LP} is approximately 
\begin{equation}\label{LPestimate}
\begin{cases}
	\frac{q-2}{r}\frac{\phi(n)}2, & 2\nmid n,\\[4pt]
	\frac{q^2-1-(q+1)}{2r}\frac{\phi(n)}2, & 2\mid n,
\end{cases}
\end{equation}
where $q=p^r$ and $p=\mathrm{char}(\F_q)$. 
Consequently, we derive an upper bound on inequivalent $\mathcal{C}_f$ of family \ref{family:LP} in Table I and on the number of $\PGaL$-inequivalent linear sets associated with these scattered polynomials. Actually, the precise value of this number could be  smaller; see \cite[Section 3]{csajbok_new_maximum_2018} for $n=6,8$.

Finally, regarding \ref{family:LZ} in Table \ref{table:numbers}, for fixed $q$ and $n$, there are exactly $\phi(n)/2$ inequivalent MRD-codes derived from this family of scattered polynomials; see \cite[Theorem 5.4]{longobardi_familyOfLinearMRD_2021}. The precise number of $\PGaL$-inequivalent linear sets provided by it is still unknown, but it is obviously smaller than or equal to $\phi(n)/2$.

\section{Construction}\label{sec:construction}
In this section, our goal is to prove the following main result.
\begin{theorem}\label{t:main}
	Let $n=2t$, $t\geq 3$ and let $q$ be an odd prime power. For each $h \in \Fqn \setminus \Fqt$ such that $h^{q^t+1}=-1$, the $\Fq$-linearized polynomial 
	\begin{equation}\label{scattered}
		\psi_{h,t}(x)=  x^{q}+x^{q^{t-1}}-h^{1-q^{t+1}}x^{q^{t+1}}+h^{1-q^{2t-1}}x^{q^{2t-1}} \in \F_{q^{n}}[x]
	\end{equation}
	is scattered.
\end{theorem}

First we note that if $t=3$ in Theorem \ref{t:main}, then the scattered polynomials $\psi_{h,3}(x)$ are exactly those constructed by Bartoli,  Zanella and Zullo in \cite{bartoli_new_maximum_2020}.

Furthermore, if we allowed $h \in \Fqt$, since $-1=h^{q^t+1}=h^2$, then $h \in \F_{q^2}$. Then we may distinguish two cases:
\begin{enumerate}[label=(\alph*)]
    \item $q \equiv 1 \pmod 4$. In this case $h \in \Fq$ and $\psi_{h,t}(x)$ becomes
        $$x^{q}+x^{q^{t-1}}-x^{q^{t+1}}+x^{q^{2t-1}}.$$ This polynomial was proven to be scattered  for each $t \geq 3$ in \cite{longobardi_familyOfLinearMRD_2021}. 
	\item $q \equiv 3 \pmod 4$. In this case $h\in\F_{q^2}$ and $h^q=-h$; hence, $t$ must be even and $\psi_{h,t}(x)$ becomes
		$$x^{q}+x^{q^{t-1}}+x^{q^{t+1}}-x^{q^{2t-1}}.$$
	This polynomial was proven in turn to be scattered in  \cite{longobardi_familyOfLinearMRD_2021}.
\end{enumerate}

Now we note that polynomials described in \eqref{scattered} can be rewritten in the following fashion: 
\begin{equation}\label{psi-as-sum}
    \psi_{h,t}(x)=L(x)+M(x),
\end{equation}
where $L(x)=x^q-h^{1-q^{t+1}}x^{q^{t+1}}$ and $M(x)=x^{q^{t-1}}+h^{1-q^{2t-1}}x^{q^{2t-1}}$.

It is straightforward to see that $L(x)$ and $M(x)$ are $\F_{q^t}$-semilinear maps of $\F_{q^{n}}$ with companion automorphisms $x \mapsto x^q$ and $x \mapsto x^{q^{t-1}}$, respectively. Moreover, we have that
\begin{equation}\label{kerL}
\ker L = \{ x \in \F_{q^n} \,:\,  x-h^{q^{2t-1}-q^t}x^{q^t}= 0\}
\end{equation}
and similarly
\begin{equation}\label{kerM}
    \ker M = \{x \in \F_{q^n} \,:\,  x+h^{q^{t+1}-q^t}x^{q^t}= 0\}.
\end{equation}
In addition, since $h^{q^t+1}=-1$, we have
\begin{align*}
 L(x)^{q^t}&=(x^q-h^{1-q^{t+1}}x^{q^{t+1}})^{q^t}=x^{q^{t+1}}-h^{q^t-q}x^q\\
 &= -h^{q^t-q}(x^q-h^{1-q^{t+1}}x^{q^{t+1}})=-h^{q^t-q}L(x)
\end{align*}
and similarly, we may prove that $M(x)^{q^t}=h^{q^t-q^{t-1}}M(x)$.
Hence, we obtain that
\begin{equation}\label{imL}
    \im \, L = \{ z \in \F_{q^{2t}} \,:\,   z^{q^t}+h^{q^t-q}z= 0 \}
    \end{equation}
    and 
    \begin{equation}\label{imM}
   \im \, M = \{ z \in \F_{q^{2t}} \,: \,   z^{q^t}-h^{q^t-q^{t-1}}z= 0 \}.
    \end{equation}
Clearly, the sets in \eqref{kerL}, \eqref{kerM}, \eqref{imL} and \eqref{imM} are 1-dimensional $\Fqt$-subspaces of $\Fqn$. 

\begin{proposition} \label{p:hcondition}
	Let $n=2t$, $t\geq 1$ and  let $h$ be in $\F_{q^{2t}}$ such that $h^{q^t+1}=-1$. Then $h^{q^2+1}\neq 1$ and $h^{q^{t-2}}\neq -h$.
\end{proposition}
\begin{proof}
	First, as $q$ is odd, $\gcd(q^2+1, q^{2t}-1)=2$ if $t$ is odd, and 
	\begin{align*}
		\gcd(q^2+1, q^t+1) &=\begin{cases}
			2, & t\equiv 0 \pmod{4};\\
			q^2+1, & t\equiv 2 \pmod{4}.
		\end{cases}
	\end{align*}

	Assume on the contrary that $h^{q^2+1}=1$. Together with $h^{q^t+1}=-1$ and the above GCD conditions, we deduce the following results.
	
	If $t$ is odd,  then $h^2=1$ which contradicts $h^{q^t+1}=-1$. If $4\mid t$, then $h^2=-1$ which implies $h^{q^2+1}=-1$ contradicting the assumption. If $2\mid t$ but $4\nmid t$, then $h^{q^t+1}=1$ contradicting $h^{q^t+1}=-1$. 
	
	From $h^{q^t+1}=-1$ and $h^{q^2+1}\neq 1$, we finally derive $h^{q^{t-2}}\neq -h$ directly.
\end{proof}

\begin{proposition}\label{directsum}
Let $n=2t$ with $t \geq 3$. The finite field $\Fqn$, seen as $\Fqt$-vector space, is both the direct sum of $\ker L$ and  $\ker M$, and of $\im\,L$ and  $\im\,M$.
\end{proposition}
\begin{proof}
 Since $\ker L$ and $\ker M$ are $1$-dimensional $\Fqt$-subspaces of $\Fqn$; it is enough to prove that  $\ker L \cap \ker M = \{0\}$. In this regard, let $u \in \ker L \cap \ker M$. By \eqref{kerL} and \eqref{kerM}, we get 
 $h^{q^{2t-1}}=-h^{q^{t+1}}$, i.e. $(h^{q^{t-2}})^{q^{t+1}}=-h^{q^{t+1}}$.
This implies that $h^{q^{t-2}}=-h$, and since $h^{{q^t}+1}=-1$ this contradicts Proposition \ref{p:hcondition}. 

Taking into account (\ref{imL}) and (\ref{imM}), a similar argument shows that the additive group of $\Fqn$, seen as $\Fqt$-vector space, can be also written as $\im L \oplus \im M$. 
\end{proof}

Consider now the following $\Fqt$-linear maps of $\Fqn$
\begin{equation}\label{R-T-maps}
    R(x)=x^{q^t}+h^{q^{t-1}-q}x \quad\textnormal{and}\quad T(x)=x^{q^t}+h^{q-q^{t-1}}x.
\end{equation}
It is straightforward to see that $\dim_{\Fqt} \ker R = \dim_{\Fqt}\ker T= 1$; moreover, $\ker T = h^{q^{t-1}-q} \ker R$.

\begin{lemma}\label{bases+split}
Let $\rho, \tau \in \F^*_{q^n}$, $n=2t$ and $t \geq 3$ such that $\rho \in \ker  R$ and $\tau \in \ker T$.
Then
\begin{enumerate}[label=(\roman*)]
    \item $\{1,\rho\}$ and $\{1,\tau\}$ are $\Fqt$-bases of $\Fqn$.
    \item if $\tau=h^{q^{t-1}-q} \rho$ and an element $\gamma \in \Fqn$ has components ($\lambda_1$,$\mu_1$) in the $\F_{q^t}$-basis $\{1,\rho\}$, then the components of  $\gamma$ in $\{1,\tau\}$ are
    \begin{equation}\label{components}
        \left(\lambda_1+ \mu_1\rho\left(1-h^{q^{t-1}-q}\right), \mu_1\right).
    \end{equation}
\end{enumerate}
\end{lemma}
\begin{proof}
$(i)$ It is enough showing that $\rho$ and $\tau$ are not in $\Fqt$. We will show that $\rho \not \in \Fqt$. A similar argument can be applied to $\tau$ as well.
Suppose that $\rho \in \Fqt$, then $\rho^{q^t-1}=1$. Then, by hypothesis,
$$1=\rho^{q^t-1}=-h^{q^{t-1}-q}.$$
Hence $h^{q^{t-2}}=-h$ which, by Proposition \ref{p:hcondition}, is not the case.\\
$(ii)$ Let $\gamma  \in \Fqn$ and suppose that $\gamma=\lambda_1+\mu_1\rho$ with $\lambda_1,\mu_1 \in \Fqt$. Also, denote by  $\lambda_2$ and $\mu_2$ the components of $\gamma$ in the  $\F_{q^t}$-basis $\{1,\tau\}$. Of course, we have
\begin{equation}\label{coordinates}
\lambda_2+\mu_2\tau=\lambda_1+\mu_1\rho
\end{equation}
Raising \eqref{coordinates} to the $q^t$-th power, and taking into account that $\rho \in \ker R$ and $\tau \in \ker T$, we get the following linear system  in the unknowns $\lambda_2$ and $\mu_2$
\begin{equation*}\label{system_coordinates}
\begin{cases}
\lambda_2+\mu_2\tau=\lambda_1+\mu_1\rho \\
\lambda_2 -\mu_2h^{q-q^{t-1}} \tau=\lambda_1-\mu_1h^{q^{t-1}-q}\rho.
\end{cases}
\end{equation*}
Clearly, this linear system has a unique solution; i.e.,
\begin{equation*}
\lambda_2=\lambda_1+\mu_1\rho \left (1-h^{q^{t-1}-q}\right ) \quad\textnormal{and} \quad
\mu_2=\mu_1.
\end{equation*} Hence, the assertion follows.
\end{proof}


\begin{proposition}\label{product-a}
For any nonzero vectors $u \in \ker L$, $v \in \ker M$ and any $a \in \Fqn$, the following statements are equivalent:
\begin{enumerate}[label=(\roman*)]
    \item  $a \in \ker R$;
    \item $av \in \ker L$;
    \item $a M(u) \in \im\,L$.
\end{enumerate}
\end{proposition}

\begin{proof}
Clearly, if $a$ is zero the statement is trivially verified. Suppose that $a \in \Fqn^*$. Let $\rho$ be a nonzero vector in $\ker R$ which means $\ker R=\langle \rho \rangle_{q^t}$.\medskip\\
$(i) \Rightarrow (ii)$. Let $a \in \langle \rho \rangle_{q^t} $, then there exists $\lambda \in \Fqt$ such that $a=\lambda \rho$. 
Then
\begin{equation*}
    L(av)=\lambda^qL(\rho v)=\lambda^q \left((\rho v)^q - h^{1-q^{t+1}}(\rho v)^{q^{t+1}} \right).
\end{equation*}
Since  $\rho \in \ker R$ and $v \in \ker M$, by \eqref{kerM} and \eqref{R-T-maps}, we get
\begin{equation*}
    (\lambda\rho v)^q\left(1-h^{-q^{t+2}+1+q^t-q^2}\right).
\end{equation*}
Moreover, since $h^{q^t+1}=-1$, the latter expression is equal to 0; hence, $av \in \ker L$.\medskip\\
$(ii)\Rightarrow (iii)$ Let $v \in \ker M$. Since $av \in \ker L$, 
\[0=L(av)=(av)^q-h^{1-q^{t+1}}(av)^{q^{t+1}}=v^q\left(a^q+h^{1-q^{t+2}}a^{q^{t+1}}\right),\]
this is 
\begin{equation}\label{a-condition}
a+h^{q^{2t-1}-q^{t+1}}a^{q^t}=0.
\end{equation}
We will prove that $aM(u) \in \im L$. So, putting $z=M(u)$,  by \eqref{imL}, this is equivalent to prove that 
\begin{equation*}
    (az)^{q^t}+h^{q^t-q}(az)=0.
\end{equation*}
By Proposition \ref{directsum}, since $u\in\ker L$, $z=M(u)\ne 0$. Also, since $h^{q^t+1}=-1$, by \eqref{imM}, we have
\begin{equation*}
    (az)^{q^t}+h^{q^t-q}(az)=zh^{q^t-q}\cdot\left(a^{q^t}h^{q^{2t-1}-q^{t+1}}+a\right)
\end{equation*}
and by \eqref{a-condition} the last expression equals 0, proving the result.\medskip\\
$(iii) \Rightarrow (i)$ As before, by Proposition \ref{directsum}, $z=M(u)$ is a nonzero element of $\im\,M$. Since $az \in \im\,L$,  by \eqref{imL} and \eqref{imM}, we obtain
\begin{equation*}
    0=(az)^{q^t}+h^{q^t-q}(az)=z\left(a^{q^t}h^{q^t-q^{t-1}}+h^{q^t-q}a\right)
=h^{q^t-q^{t-1}}z\left(a^{q^t}+h^{q^{t-1}-q}a\right),
\end{equation*}
which implies $a^{q^t}+h^{q^{t-1}-q}a=0$. Then, by \eqref{R-T-maps}, $a \in \ker R$. Finally, since $\ker R$ is a $1$-dimensional $\Fqt$-subspace of $\Fqn$, $a=\lambda \rho$ for some $\lambda \in \Fqt$. 
\end{proof}

Similarly, we have the following result.

\begin{proposition}\label{product-b}
For any nonzero vectors $u \in \ker L$, $v \in \ker M$ and any $b \in \Fqn$, the following statements are equivalent:
\begin{enumerate}[label=(\roman*)]
    \item $b \in \ker T$;
    \item $b\,u \in \ker M$;
    \item $ b L(v) \in \im\, M$.
\end{enumerate}
\end{proposition}
\begin{proof}
As before, we may suppose that $b \in \Fqn^*$, otherwise the statement is easily verified. Let $\tau$ be a nonzero vector in $\ker T$ which means $\ker T=\langle \tau \rangle_{q^t}$.\medskip \\ 
$(i) \Rightarrow (ii)$. Let $b \in \langle \tau\rangle_{q^t} $, then there exists $\lambda \in \Fqt$ such that $b=\lambda \tau$. Consider $M(bu)=M(\lambda \tau u)$. 
Then
\begin{equation*}
    M(bu)=\lambda^{q^{t-1}}M(\tau u)=\lambda^{q^{t-1}} \left((\tau u)^{q^{t-1}} + h^{1-q^{2t-1}}(\tau u)^{q^{2t-1}} \right).
\end{equation*}
Since  $\tau \in \ker T$, $u \in \ker L$, by \eqref{kerL} and \eqref{R-T-maps}, we get
\begin{equation*}
    (\lambda\tau u)^{q^{t-1}}\left(1-h^{1-q^{t-2}+q^t-q^{2t-2}}\right).
\end{equation*}
Moreover, since $h^{q^t+1}=-1$, the latter expression is equal to 0; hence, $bu \in \ker M$.\medskip\\
$(ii)\Rightarrow (iii)$ Let $u \in \ker L$. Since $bu \in \ker M$, 
\[0=M(bu)=(bu)^{q^{t-1}}+h^{1-q^{2t-1}}(bu)^{q^{2t-1}}=u^{q^{t-1}}(b^{q^{t-1}}+h^{1-q^{t-2}}b^{q^{2t-1}}),\]
this is 
\begin{equation}\label{b-condition}
b^{q^t}+h^{q-q^{t-1}}b=0.
\end{equation}
We will prove that $bL(v) \in \im M$. So, putting $z=L(v)$,  by \eqref{imM}, this is equivalent to prove that 
\begin{equation*}
    (bz)^{q^t}-h^{q^t-q^{t-1}}(bz)=0.
\end{equation*}
By Proposition \ref{directsum}, since $v\in\ker M$ then $z=L(v) \in \im\,L   $ and $L(v)\ne 0$. Since $h^{q^t+1}=-1$, by \eqref{imL} we get
\begin{equation*}
    (bz)^{q^t}-h^{q^t-q^{t-1}}(bz)=-zh^{q^t-q}(b^{q^t}+h^{q-q^{t-1}}b)
\end{equation*}
which equals 0 taking \eqref{b-condition} into account, proving the result.\medskip\\
$(iii) \Rightarrow (i)$ As before, by Proposition \ref{directsum}, $z=L(v)$ is a nonzero element of $\im\,L$. Since $bz \in \im\,M$,  by \eqref{imL} and \eqref{imM}, we obtain
\begin{equation*}
    0=(bz)^{q^t}-h^{q^t-q^{t-1}}(bz)=-zh^{q^t-q}(b^{q^t}+h^{q-q^{t-1}}b),
\end{equation*}
this is $b^{q^t}+h^{q-q^{t-1}}b=0$. By \eqref{R-T-maps}, then $b \in \ker T$. Since $\ker T$ is a $1$-dimensional $\Fqt$-subspace of $\Fqn$, $b=\lambda \tau$ for some $\lambda \in \Fqt$. 
\end{proof}

We are now in the position to prove our main result of this section.
\begin{proof}[Proof of Theorem \ref{t:main}]\label{scatteredconditio}
Let $\psi(x):=\psi_{h,t}(x)$, we want to prove that for each $x\in\F_{q^n}^*$ and for each $\gamma\in\F_{q^n}$ such that 
\begin{equation}\label{e:psirx}
\psi(\gamma x)=\gamma \psi(x)
\end{equation}
we get $\gamma \in\Fq$.
Recall that 
\[\psi(x)=L(x)+M(x)\] 
as in \eqref{psi-as-sum}.
Also, by Proposition \ref{directsum}, 
any $x \in\Fqn$ can be uniquely written as $x=x_1+x_2$, where $x_1 \in \ker L$ and $x_2 \in \ker M$. Similarly, by Lemma \ref{bases+split}, if $\gamma \in\Fqn$ there are exactly two elements $\lambda_1,\mu_1 \in \Fqt$ and two elements $\lambda_2, \mu_2 \in \Fqt$  such that $$\lambda_1+\mu_1\rho=\gamma= \lambda_2+\mu_2\tau$$ where
$\rho \in \ker R$ and $\tau= h^{q^{t-1}-q}\rho$. It is easy to check that $\tau\in\ker T$. Putting $a=\mu_1\rho$ and $b=\mu_2\tau$, which imply $a\in \ker R$ and $b\in \ker T$,  Condition \eqref{e:psirx} may be re-written as follows
\begin{equation}\label{notsplit}
\begin{split}
L((\lambda_1+a)(x_1+x_2))+&M((\lambda_2+b)(x_1+x_2))=\\
& = (\lambda_2+b)L(x_1+x_2)+(\lambda_1+a)M(x_1+x_2).
\end{split}
\end{equation}
Also, since $x_1 \in \ker L$, $x_2 \in \ker M$, $L(x)$ and $M(x)$ are $\Fqt$-semilinear maps and by $(ii)$ of Propositions \ref{product-a} and $(ii)$ of Proposition \ref{product-b}, Equation \eqref{notsplit} is equivalent to
\begin{equation*}
\begin{split}
 L(\lambda_1 x_2)+L(a x_1)&+M(\lambda_2x_1) +M(b x_2)=\\
& =\lambda_2L(x_2)+b L(x_2)+\lambda_1M(x_1)+a M(x_1).
\end{split}
\end{equation*}
and hence
\begin{equation}\label{splitted}
\begin{split}
 \lambda_1 ^qL(x_2)+L(a x_1)&-\lambda_2L(x_2)-a M(x_1)=\\
& = b L(x_2)+\lambda_1M(x_1)-\lambda_2^{q^{t-1}}M(x_1)-M(b x_2).
\end{split}
\end{equation}
Now, since the image spaces of the maps $L(x)$ and $M(x)$ are $\Fqt$-spaces, taking $(iii)$ of Proposition \ref{product-a}  and $(iii)$ of Proposition \ref{product-b} into account, the expressions on  left and right hand sides of \eqref{splitted} belong to $\im \, L$ and $\im \, M$, respectively.
By Proposition \ref{directsum}, both sides of \eqref{splitted} must be equal to zero an hence we obtain the following system
\begin{equation*}
\begin{cases}
L(ax_1)-aM(x_1)=(\lambda_2-\lambda^q_1)L(x_2)\\
bL(x_2)-M(bx_2)=(\lambda_2^{q^{t-1}}-\lambda_1)M(x_1).
\end{cases}
\end{equation*}
Raising to the $q$-th power the second equation, we get
\begin{equation}\label{system1}
\begin{cases}
L(ax_1)-aM(x_1)=(\lambda_2-\lambda^q_1)L(x_2)\\
b^qL(x_2)^q-M(bx_2)^q=(\lambda_2-\lambda_1^q)M(x_1)^q.
\end{cases}
\end{equation}
Since $a=\mu_1 \rho$, $b=\mu_2\tau$ and $\tau=h^{q^{t-1}-q}\rho$, from Lemma \ref{bases+split} it follows that $\mu_1=\mu_2$ and $b=h^{q^{t-1}-q}a$. 

If $a=0$, we have $\mu_1=0$ and hence $\gamma=\lambda_1=\lambda_2\in\F_{q^t}$. Also, from \eqref{system1}, if $\lambda_2\neq \lambda_1^q$, then $L(x_2)=M(x_1)=0$. By Proposition \ref{directsum}, $x=x_1=x_2=0$, a contradiction. Then $\lambda_1=\lambda_2=\lambda_1^q$, which gives $\lambda_1\in\F_q$, i.e. $\gamma \in\F_q$ and $\psi(x)$ is a scattered polynomial.

In the remainder of the proof, we are going to show that $a\neq 0$, i.e. $\gamma\in\F_{q^{2t}}\setminus\F_{q^t}$ leads to contradictions. Depending on the value of $x_1$ and $x_2$, we separate the proof into three cases.
\medskip

\noindent\textbf{Case 1.} $x_1=0$. The system in \eqref{system1} is reduced to
    \begin{equation*}
	\begin{cases}\label{x_1=0}
	(\lambda_2-\lambda^q_1)L(x_2)=0\\
	b^qL(x_2)^q-M(bx_2)^q=0.
	\end{cases}
	    \end{equation*}
	Since $x_2 \in \ker M$, by \eqref{R-T-maps}, from second equation we get
	\begin{equation*}
	    b^{q-1}=\frac{h^{q-1}}{\left(x_2(1+h^{q-q^{t-1}})\right)^{q^2-1}}.
	\end{equation*}
	Then, there exists $\lambda \in \Fq^*$ such that 
	\begin{equation*}
	    b=\lambda \cdot \frac{h}{\left(x_2(1+h^{q-q^{t-1}})\right)^{q+1}}.
	\end{equation*}
	Since $b \in \ker T$, then $b^{q^t}+h^{q-q^{t-1}}b=0$ and we get
	\begin{equation*}
	\frac{h^{q^t-1}}{\left(x_2(1+h^{q-q^{t-1}})\right)^{q^t(q+1)}} +   \frac{h^{q-q^{t-1}}}{\left(x_2(1+h^{q-q^{t-1}})\right)^{q+1}}=0,
	\end{equation*}
	whence, since $x_2 \in \ker M$,
	\begin{equation*}
	\Biggl( \frac{1+h^{q-q^{t-1}}}{
	h^{q-1}(1+h^{q^{t-1}-q})} \Biggr )^{q+1}=-h^{1+q-q^{t-1}-q^t}.
	\end{equation*}
	This is equivalent to
	\begin{equation*}
	\Biggl( \frac{h^{q^{t-1}-1}(1+h^{q-q^{t-1}})}{
	h^{q-1}(1+h^{q^{t-1}-q})} \Biggr )^{q+1}=-1,
	\end{equation*}
	whence we have $1^{q+1}=-1$, a contradiction.

\medskip
\noindent\textbf{Case 2.} $x_2=0$. The system in \eqref{system1} is reduced to
	   \begin{equation*}
	\begin{cases}
	L(ax_1)-aM(x_1)=0\\
	(\lambda_2-\lambda_1^q)M(x_1)^q=0.
	\end{cases}
	    \end{equation*}
	By the first equation, taking into account that $a\in  \ker R$, $x_1 \in  \ker L$ and $h^{q^t+1}=-1$, we obtain
	\begin{equation*}
	    a^{q-1}=(x^{q}_1)^{q^{t-2}-1}\cdot \frac{1+h^{1-q^{t-2}}}{1+h^{q^{t+2}-1}}=\left(x^{q}_1(1+h^{q^{t+2}-1})\right)^{q^{t-2}-1}.
	\end{equation*}
	Then there exists $\lambda \in \Fq^*$ such that
	\begin{equation*}
	a=\lambda \left(x^q_1(1+h^{q^{t+2}-1})\right)^\nu,
	\end{equation*}
	where $\nu=(q^{t-2}-1)/(q-1)$.
	
	By \eqref{R-T-maps},  since $a \in \ker R$, then
	\begin{equation*}
	    \left(x_1^{q^{t+1}}(1+h^{q^2-q^{t}})\right)^\nu+h^{q^{t-1}-q}\left(x_1^q(1+h^{q^{t+2}-1})\right)^\nu=0.
	\end{equation*}
	Moreover, since $x_1 \in \ker L$, then
	\begin{equation*}
	     \Biggl (\ \frac{h^{q^{t+1}-1} (1+h^{q^2-q^t})}{1+h^{q^t-q^2}}  \Biggr)^\nu=-h^{q^{t-1}-q}.
	\end{equation*}
	The last expression is equivalent to
	  \[ \Biggl (\ \frac{h^{q^{t}-q} (1+h^{q^2-q^t})}{h^{q^2-q}(1+h^{q^t-q^2})}  \Biggr)^\nu=-1, \]
	whence $1^\nu=-1$, leading to a contradiction.

\medskip
\noindent\textbf{Case 3.} $x_1,x_2 \neq 0$. Recall that $a \in \ker R$, $b=h^{q^{t-1}-q}a$, $\lambda_2 = \lambda_1+(1-h^{q^{t-1}-q})a$, $x_1 \in \ker L$ and $x_2 \in \ker M$. Then, by \eqref{system1}, $a$ turns out to be a nonzero solution of the following linear system

\begin{equation}\label{complete-linear-system}
    \begin{cases}
x_1^q(1+h^{q^{t}-q^2})a^q-\left(M(x_1)+(1-h^{q^{t-1}-q})L(x_2)\right)a=(\lambda_1-\lambda^q_1)L(x_2)\\
h^{q^t-q^2}L(x_2)^qa^q+\left(x_2^{q^t}(1+h^{q^{t-1}-q})-(1-h^{q^{t-1}-q})M(x_1)^q\right)a=(\lambda_1-\lambda_1^q)M(x_1)^q.
\end{cases}
\end{equation}

By $x_1\in \ker L$ and $x_2\in \ker M$, we obtain the following two equations which will be frequently used later,
\begin{align*}
	L(x_2) &= x_2^q(1+h^{1-q^{t+2}}),\\
	M(x_1) &= x_1^{q^{t-1}} (1+h^{1-q^{t-2}}).
\end{align*}

\medskip
\noindent\textit{- Case 3.1} First of all, suppose that $\lambda_1 \in \Fq$, then System \eqref{complete-linear-system} becomes

\begin{equation}\label{eq:case3.1_star}
    \begin{cases}
x_1^q(1+h^{q^{t}-q^2})a^q-\left(M(x_1)+(1-h^{q^{t-1}-q})L(x_2)\right)a=0\\
h^{q^t-q^2}L(x_2)^qa^q+\left(x_2^{q^t}(1+h^{q^{t-1}-q})-(1-h^{q^{t-1}-q})M(x_1)^q\right)a=0.
\end{cases}
\end{equation}
and since $a$ is a nonzero solution then
\begin{eqnarray}\label{determinant-incomplete}
 x_1^q(1+h^{q^{t}-q^2})\left(x_2^{q^t}(1+h^{q^{t-1}-q})-(1-h^{q^{t-1}-q})M(x_1)^q\right)=\nonumber\\=
 -h^{q^t-q^2}L(x_2)^q\left(M(x_1)+(1-h^{q^{t-1}-q})L(x_2)\right).
\end{eqnarray}

Since $L(x_2) \neq 0 \neq M(x_1)$, from \eqref{eq:case3.1_star} we get
$$M(x_1)^q\left(x_1^q(1+h^{q^{t}-q^2})a^q-M(x_1)a\right)=L(x_2)\left(h^{q^t-q^2}L(x_2)^q a^q+x_2^{q^t}(1+h^{q^{t-1}-q})a\right)$$
whence
\begin{align}
\nonumber \left(x_1^q M(x_1)^q(1+h^{q^{t}-q^2}) - h^{q^t-q^2}L(x_2)L(x_2)^q\right)a^q=\\
=\left(M(x_1)M(x_1)^q+x_2^{q^t}(1+h^{q^{t-1}-q})L(x_2)\right)a. \label{Fq-case}
\end{align}

Next we want to show that the coefficient of $a^q$ in \eqref{Fq-case} cannot be $0$. By way of contradiction, suppose that
\begin{equation}\label{D2}
	x_1^q\left(1+h^{q^{t}-q^2}\right)M(x_1)^q=h^{q^t-q^2}L(x_2)L(x_2)^q,
\end{equation}
from \eqref{determinant-incomplete} it follows
 \begin{equation}\label{D1}
 x_1^q x_2^{q^t}\left(1+h^{q^{t}-q^2}\right)\left(1+h^{q^{t-1}-q}\right)=-h^{q^t-q^2}L(x_2)^qM(x_1).
 \end{equation}
 Since $x_1 \in \ker L$ and $x_2 \in \ker M$, this is equivalent to

 \begin{equation}\label{x1fromx2}
   x_1^{q^{t-2}-1}=-(x_2^q)^{q^{t-2}-1}\frac{(1+h^{q-q^{t-1}})(1+h^{q^{t-2}-1})}{(1+h^{1-q^{t+2}})(1+h^{q^{2t-1}-q^{t-3}})}.
\end{equation}
This formula is equivalent to
\begin{equation}
\label{newformula}
    x_1^{q^{t-2}-1}=-\Biggl (x_2^q\cdot \frac{1+h^{1-q^{t+2}}}{(1+h^{q-q^{t-1}} )h^{q^{t-1}}} \Biggr )^{q^{t-2}-1}.
\end{equation}
Since \begin{equation*}
 d=\gcd\left(2t,t-2\right)=\gcd (4,t-2)=\left\{\begin{array}{lll}
 1, &\mbox{if $t$ odd}\\
 2,  &\mbox{if $t \equiv 0 \pmod 4$ }\\
 4,  &\mbox{if $t \equiv 2 \pmod 4$ }
 \end{array}
 \right.  
 \end{equation*}
 there exists a solution  of the equation $x^{q^{t-2}-1}=-1$ for any $t \geq 3$ and $t \not \equiv 2 \pmod 4$ in $\Fqn$. Thus, for $t\equiv 2 \pmod 4$, Equation \eqref{newformula} gives a contradiction. In the remaining cases,
 \begin{equation*}
 x_1= \omega x^q_2 \cdot \frac{1+h^{1-q^{t+2}}}{(1+h^{q-q^{t-1}} )h^{q^{t-1}}},
 \end{equation*}
for some $\omega\in \Fqn^*$ satisfying $\omega^{q^{t-2}-1}=-1$.
By substituting this expression in \eqref{D2}, since $x_1\in \ker L$, we get
\begin{equation}\label{eq:omega}
\omega^{q+1}=-1,
\end{equation}
and hence $\omega\in\F_{q^2}$.

If $t$ is odd, since $\omega^{q^{t-2}}=-\omega$, then $\omega^q=-\omega$ and from \eqref{eq:omega}, we get $\omega=\pm 1$. 
If $t \equiv 0 \pmod 4$, since $\omega=\omega^{q^{{t-2}}}=-\omega$. In both cases we get a contradiction.


Then, by \eqref{Fq-case}, we get 
\begin{equation}\label{eq:a^q-1}
\begin{split}
    a^{q-1}&=\frac{M(x_1)M(x_1)^q-h^{q^t-q^{t+1}}x_2L(x_2)(1+h^{q^{t-1}-q})}{x_1^q(1+h^{q^{t}-q^2})M(x_1)^q-h^{q^t-q^2}L(x_2)L(x_2)^q}=\\
    &=h^{q-1} \cdot \frac{1+h^{q^{t-1}-q}}{1+h^{q^t-q^2}} \cdot \frac{x_1M(x_1)-x_2L(x_2)}{x_1^qM(x_1)^q-x_2^qL(x_2)^q}=\\
    &=\Biggl( \frac{h}{(1+h^{q^{t-1}-q})(x_1M(x_1)-x_2L(x_2))} \Biggr)^{q-1},
\end{split}
\end{equation}
whence
\begin{equation*}
    a=\lambda \cdot \frac{h}{(1+h^{q^{t-1}-q})(x_1M(x_1)-x_2L(x_2))}
\end{equation*}
for some $\lambda \in \F_{q}^*$. Since $a \in \ker R$, then
\begin{equation*}
\frac{h^{q^t}}{\left((1+h^{q^{t-1}-q})(x_1M(x_1)-x_2L(x_2))\right)^{q^t}}+ h^{q^{t-1}-q}\cdot \frac{h}{(1+h^{q^{t-1}-q})(x_1M(x_1)-x_2L(x_2))}=0.
\end{equation*}
Recalling that $x_1 \in \ker L$ and $x_2 \in \ker M$, we get
\[x_1M(x_1)-x_2L(x_2) = x_1^{q^{t-1}+1}\left(1+h^{1-q^{t-2}}\right)-x_2^{q+1}\left(1+h^{1-q^{t+2}}\right).  \]
This implies that
\begin{equation*}
    -\frac{1}{h^{q^t}(1+h^{q-q^{t-1}})}+\frac{h^{q^{t-1}-q}\cdot h}{1+h^{q^{t-1}-q}}=0,
\end{equation*}
which means $h^{q^t+1}=1$, a contradiction. Then $\lambda_1$ may not belong to $\Fq$.
\medskip

\noindent- $\textit{Case 3.2}$ Let $\lambda_1\notin\F_q$ and let $a$ be a nonzero solution of System \eqref{complete-linear-system}. 
If this system admits more than one solution, then each $2\times 2$ minor of the associated matrix of \eqref{complete-linear-system} is zero. In particular Equations \eqref{D2} and \eqref{D1} hold true, obtaining a contradiction as in the previous case.

Then, System \eqref{complete-linear-system} must admits a unique nonzero solution $(a,a^q)  \in \F_{q^n}^2$. By computing the ratio $a^{q-1}$ of its components, we get

\begin{equation*}
\begin{split}
    a^{q-1}&=\frac{\begin{vmatrix}
   L(x_2) &  -M(x_1) \\
    M(x_1)^q &  x_2^{q^t}(1+h^{q^{t-1}-q}) \end{vmatrix}}{\begin{vmatrix}
    x_1^q(1+h^{q^{t}-q^2}) & L(x_2) \\
    h^{q^t-q^2}L(x_2)^q & M(x_1)^q
    \end{vmatrix}}=\\
    &=\frac{M(x_1)M(x_1)^q-h^{q^t-q^{t+1}}x_2L(x_2)(1+h^{q^{t-1}-q})}{x_1^q(1+h^{q^{t}-q^2})M(x_1)^q-h^{q^t-q^2}L(x_2)L(x_2)^q}.
\end{split}
\end{equation*}This is again Equation \eqref{eq:a^q-1}. 
Repeating the arguments as in \textit{Case 3.1} we get a contradiction. \qedhere
\end{proof}

\section{A new family of MRD codes}\label{sec:MRD}
Let start by the following preliminary general result.

\begin{lemma}\label{le:full_auto_MRD}
	Let $f$ be a scattered polynomial in  ${\mathcal L}_{n,q}[x]$ and let $\cC_f$ denote the associated MRD code.  Then $\Aut(\cC_f)$ consists of elements $(\alpha x^{q^\ell}, L_2, \rho)\in {\mathcal L}_{n,q}[x]\times  {\mathcal L}_{n,q}[x]\times \Aut(\F_{q^n})$ with invertible $L_2$ , $\alpha\in \F_{q^n}^*$ and $\ell \in \{0,1,\ldots, n-1\}$ such that $\cC_{f^{\rho q^\ell}} \circ x^{q^\ell}\circ L_2 = \cC_f$.
	
	Furthermore, there is a bijection between each $L_2$ and each $\GL(2,q^n)$-equivalence map from $U_f$ to $U_{f^{\rho q^\ell}}$, where $U_f =\{ (x,f(x)): x \in \F_{q^n} \}$. In particular, the multiplicative group of $I_R(\cC_f)\setminus \{0\}$ and the $\GL(2,q^n)$-automorphism group of $U_f$ are isomorphic.
\end{lemma}
\begin{proof}
	Suppose that $\varphi\in I_L(\cC_f)$ and $(L_1, L_2,\rho)\in \Aut(\cC_f)$. Then for any $g\in \cC_f$, there exists an element $g'\in \cC_f$ such that
	\[ \varphi \circ (L_1 \circ g^\rho \circ L_2) = L_1 \circ g' \circ L_2, \]
	which means 
	\[  L_1^{-1}\circ \varphi \circ L_1 \circ g^\rho\in \cC_f, \]
	 and hence, $ L_1^{-1}\circ \varphi \circ L_1 \in I_L(\cC_f)$. In another word, $L_1$ is in the normalizer of $ I_L(\cC_f)$ in $\GL(n,q)$, which actually isomorphic to $(\F^*_{q^n},\cdot)\rtimes \Aut(\F_{q^n}/\F_q)$; see \cite{liebhold_automorphism_2016} and \cite[Hilfssatz 3.11, Chapter 2]{huppert_endliche_1967}. Thus $L_1(x)=\alpha x^{q^\ell}$ for some $\alpha\in \F_{q^n}^*$ and $\ell\in \{0,1,\ldots, n-1\}$. Moreover, 
	 \[L_1 \circ \cC_{f^{\rho}}\circ L_2=\cC_{f^{\rho q^\ell}}\circ x^{q^\ell} \circ L_2 = \cC_f.\]
	 Therefore, we have finished the first part of the statement.
	 
	 As the identity map is in $\cC_{f^{\rho q^\ell}}$, $x^{q^\ell} \circ L_2(x)\in \cC_f$, that is $x^{q^\ell} \circ L_2(x)=ax+bf(x)$ for some $a,b\in \F_{q^n}$. Furthermore, $\cC_{f^{\rho q^\ell}} \circ x^{q^\ell}\circ L_2 = \cC_f$ also implies the existence of $c,d\in \F_{q^n}$ such that
	 \begin{equation}\label{eq:equiv}
	 cx+df(x)=f^{\rho q^\ell}\left(ax+bf(x)\right),
	 \end{equation}
	 for all $x\in \F_{q^n}$.
	 By setting $y=ax+bf(x)$, \eqref{eq:equiv} is equivalent to 
	 \[
	 		\begin{pmatrix}
	 a&b\\c&d
	 \end{pmatrix}
	 \begin{pmatrix}
	 x\\ f(x)
	 \end{pmatrix}=
	 \begin{pmatrix}
	 y\\ f^{\rho q^\ell }(y)
	 \end{pmatrix}.
	 \]
	 Also, the matrix 	$	\begin{pmatrix}
	 a&b\\c&d
	 \end{pmatrix}$ is invertible. Indeed, if there exists  $k\in \F_{q^n}^*$ such that $(c,d)=k(a,b)$ then by \eqref{eq:equiv} we get $f(x)=\mu x$ for some $\mu\in\F_{q^n}^*$, contradicting the fact that $f$ is scattered.
	 This means $U_f$ is $\GL(2,q^n)$-equivalent to $U_{f^{\rho q^\ell}}$. Also from \eqref{eq:equiv}, it can be easily seen that $(c,d)$ is uniquely determined by $(a,b)$. Therefore, there is a 1-1 correspondence between $L_2$ and $	\begin{pmatrix}
	 a&b\\c&d
	 \end{pmatrix}$ mapping $U_f$ to $U_{f^{\rho  q^\ell}}$. 
	 
	 In particular, when $f=f^{\rho q^\ell}$, all such $\begin{pmatrix}
	 	a&b\\c&d
	 \end{pmatrix}$ form the $\GL(2,q^n)$-automorphism group of $U_f$, denoted by $G_f$. As each nonzero element of $I_R(\cC_f)$ is invertible (see \cite{liebhold_automorphism_2016,lunardon_kernels_2017}), it is straighforward to see that the map described above determines an isomorphism between the groups $G_f$ and $I_R(\cC_f)\setminus \{0\}$. Hence, the result follows.
\end{proof}

\bigskip

Let $\psi_{h,t}$ be defined as in Theorem \ref{t:main} and 
	\begin{equation}\label{eq:C_ht}
	\cC_{h,t}=\{ax+b\psi_{h,t}(x): a,b \in \F_{q^n}\}.
	\end{equation}
	By a result in  \cite{sheekey_new_2016}, $\cC_{h,t}$ is an $\F_{q^n}$-linear MRD code.
	
	\bigskip

The following result is about the equivalence among $\cC_{h,t}$'s for different $h$ and the automorphism group of $\cC_{h,t}$.
\begin{theorem}\label{t:code_equiv}
	Let $n=2t$ with $t>4$. For each $h,k\in \F_{q^n}$ satisfying $h^{q^t+1}=k^{q^t+1}=-1$, the following hold    
	\begin{enumerate}[label=(\alph*)]
		\item if $t\not\equiv 2\pmod{4}$, then $\cC_{h,t}$ and $\cC_{k,t}$ are equivalent if and only if $h=\pm k^\rho$ where $\rho\in \Aut(\F_{q^n})$;
		\item if $t\equiv 2\pmod{4}$, then $\cC_{h,t}$ and $\cC_{k,t}$ are equivalent if and only if  $h=\ell k^\rho$ where $\ell^{q^2+1}=1$ and $\rho\in \Aut(\F_{q^n})$.
	\end{enumerate}
	The full automorphism group $\Aut(\cC_{h,t})$ is isomorphic to $(\F_{q^n}^*,\cdot )\times (\F_{q^2}^*,\cdot )\rtimes H$,
	where $H=\{\rho\in \Aut(\F_{q^n}): h=\pm h^\rho \}$ if $t\not\equiv 2\pmod{4}$, and $H=\{\rho\in \Aut(\F_{q^n}): (h^\rho/h)^{q^2+1}=1 \}$ if $t\equiv 2\pmod{4}$.	
\end{theorem}
\begin{proof}
	Let $U_h=\{(x,\psi_{h,t}(x)): x\in \F_{q^n}\}$ and $U_k=\{(x,\psi_{k,t}(x)): x\in \F_{q^n}\}$.
	By Theorem \ref{t:code_equiv=space_equiv}, we only have to consider the  $\GaL(2,q^n)$-equivalence $U_h$ and $U_k$. Therefore, to get (a) and (b), we just need to show that  a necessary and sufficient condition for  that $U_h$ is $\GL(2,q^n)$-equivalent to $U_k$ is 
	\begin{equation}\label{eq:GL(2,q^n)-equi}
	 h=\begin{cases}
		\pm k, & t\not\equiv 2\pmod{4};\\
		\ell k, & t\equiv 2\pmod{4},
	\end{cases}
	\end{equation}
	where $\ell\in \F_{q^n}$ satisfies $\ell^{q^2+1}=1$.

	Hence, we only have to consider the existence of invertible matrix 
	$\begin{pmatrix}
	a&b\\c&d
	\end{pmatrix}$ over $\F_{q^n}$ such that for each $x\in \F_{q^n}$ there exists $y\in \F_{q^n}$ satisfying
	\[
		\begin{pmatrix}
			a&b\\c&d
		\end{pmatrix}
		\begin{pmatrix}
			x\\\psi_{h,t}(x)
		\end{pmatrix}=
		\begin{pmatrix}
			y\\\psi_{k,t}(y)
		\end{pmatrix}.
	\]
	This is equivalent to
	\begin{equation}\label{eq:equiv_t>3}
		cx+d\psi_{h,t}(x)=\psi_{k,t}\left(ax+b\psi_{h,t}(x)\right),
	\end{equation}
	for all $x\in \F_{q^n}$.
	The right-hand-side of \eqref{eq:equiv_t>3} is
	\begin{align*}
		&\left(b^q+k^{1-q^{t+1}}b^{q^{t+1}} h^{q^{t+1}-q^2}\right)x^{q^2} + \left(b^{q^{t-1}}h^{q^{t-1}-q^{t-2}} + b^{q^{2t-1}}k^{1-q^{2t-1}}\right)x^{q^{t-2}}+ \\
		+&\left(b^q+b^{q^{t-1}}-h^{q^{t+1}-q^t}k^{1-q^{t+1}}b^{q^{t+1}} -k^{1-q^{2t-1}}h^{q^{2t-1}-q^t}b^{q^{2t-1}}\right)x^{q^t}+\\
		+&\left( b^{q^{t-1}}+k^{1-q^{2t-1}}h^{q^{2t-1}-q^{2t-2}}b^{q^{2t-1}} \right)x^{q^{2t-2}} - \left(h^{q-q^{t+2}}b^q+k^{1-q^{t+1}}b^{q^{t+1}}\right)x^{q^{t+2}}+\\
		+&\left(b^qh^{q-1} - b^{q^{t-1}}h^{q^{t-1}-1}-k^{1-q^{t+1}}b^{q^{t+1}}+k^{1-q^{2t-1}}b^{q^{2t-1}}\right)x+\psi_{k,t}(ax).
	\end{align*}
	
	As $t>4$, it is easy to see that the coefficients of $x^{q^2}$, $x^{q^{t-2}}$, $x^{q^t}$,  $x^{q^{t+2}}$ and $x^{q^{2t-2}}$ in the right-hand-side of \eqref{eq:equiv_t>3} must be $0$. Depending on whether the value of $b$ equals $0$ or not, we separate the proof into two cases.
	
	\textbf{Case 1.} $b\neq 0$. By the coefficient of $x^{q^2}$ (or equivalently, by the coefficient of $x^{q^{t+2}}$), we get
	\begin{equation}\label{eq:equiv_t>3_b1}
		b^{q^{t+1}-q}=-k^{-q-1}h^{q^2+q}.  
	\end{equation}
	Similarly, by the coefficient of $x^{q^{t-2}}$ (or equivalently, by the coefficient of $x^{q^{2t-2}}$), we get
	\begin{equation}\label{eq:equiv_t>3_b2}
	b^{q^{t+1}-q}=-k^{-q+q^2}h^{q-1}.  
	\end{equation}
	By \eqref{eq:equiv_t>3_b1} and \eqref{eq:equiv_t>3_b2}, we obtain
	\[ h^{q^2+1}=k^{q^2+1}.\]
	Let $\ell=h/k$. Then $\ell^{q^2+1}=1$. By the assumption that $h^{q^t+1}=k^{q^t+1}=-1$, we have $\ell^{q^t+1}=1$.  
	
	If $t\not\equiv 2 \pmod{4}$, then $\ell =\pm 1$;  if $t\equiv 2 \pmod{4}$, then we obtain $\ell^{q^2+1}=1$, which implies $\ell\in \F_{q^4}$.
	
	\textbf{Case 2.} $b=0$. If \eqref{eq:equiv_t>3} holds, then $c=0$ and 
	\begin{equation}\label{eq:equiv_t>3_last_cond}
	\begin{cases}
	d=a^q=a^{q^{t-1}}\\
	dh^{1-q^{t+1}} = k^{1-q^{t+1}} a^{q^{t+1}}\\
	dh^{1-q^{2t-1}} =k^{1-q^{2t-1}}a^{q^{2t-1}}.
	\end{cases}	
	\end{equation}
	The first equation in \eqref{eq:equiv_t>3_last_cond} implies that $a \in \F_{q^{\gcd(2t,t-2)}}$ which means $a\in\F_{q^{\gcd(t-2, 4)}}$. Let $\ell=h/k$. The last two equations in \eqref{eq:equiv_t>3_last_cond} become
	\begin{equation}\label{eq:equiv_t>3_lastlast_cond}
	d\ell^{1-q^{t+1}}=a^{q^3}=d\ell^{1-q^{2t-1}}.
	\end{equation}
	Thus $\ell^{q^{2t-1}-q^{t+1}}=1$. This means $\ell \in \F_{q^{\gcd(t-2, 2t)}}=\F_{q^{\gcd(t-2, 4)}}$. By the assumption that $h^{q^t+1}=k^{q^t+1}=-1$, we have $\ell^{q^t+1}=1$. If $\gcd(t-2,4)\in\{1,2\}$, then $\ell =\pm 1$;  if $\gcd(t-2,4)=4$, then we obtain $\ell^{q^2+1}=1$.

	Next, let us handle the case $t\equiv 2\pmod{4}$. By \eqref{eq:equiv_t>3_lastlast_cond}, 
	\[
	a^{q^3-q}=\ell^{1-q^{t+1}}=\ell^{1-q^3}=\ell^{(1-q)(q^2+q+1)}=\ell^{q-q^2}=\ell^{q+1}.
	\]
	For a given $\ell\in \F_{q^4}$, we can always find $a\in \F_{q^4}$ satisfying the above equation. Moreover, it is routine to verify that such $a$ satisfies \eqref{eq:equiv_t>3_last_cond} provided that $h=\ell k$ with $\ell^{q^2+1}=1$; note here that by \eqref{eq:equiv_t>3_last_cond} $d$ depends on $a$. This complete the proof of equivalence between $\cC_{h,t}$ and $\cC_{k,t}$.
	
	Finally, we determine  the automorphism group of $\cC_{h,t}$. By Lemma \ref{le:full_auto_MRD}, we only have to determine all $\sigma\in \Aut(\F_{q^n})$ and all the elements in $\GL(2,q^n)$ mapping $U_h$ to $U_{h^\sigma}$. By (a) and (b), we see that $\sigma$ must satisfy the condition $h^{\sigma}=\pm h$ for $t\not\equiv 2\pmod{4}$, and $ (h^\sigma /h)^{q^2+1}=1$ for $t\equiv 2\pmod{4}$. All such $\sigma$ form a subgroup $H\subseteq\Aut(\F_{q^n})$.
	To accomplish the proof, we  just need to continue the computation of the first part for $k=h$ and determine which matrix $\begin{pmatrix}
		a&b\\c&d
		\end{pmatrix}$ defines an equivalence map from $U_h$ to itself.
Depending on whether $b=0$ or not, we consider two cases.
	
	When $b=0$, we only have to let $k=h$ in \eqref{eq:equiv_t>3_last_cond}. From there we derive that $d=a^q=a^{q^{t-1}}=a^{q^{t+1}}$.  Therefore $a\in \F_{q^{\gcd(t,2)}}^*$.
	
	When $b\neq 0$, by letting $k=h$, we see that the coefficient of $x^{q^t}$ in the right-hand-side of \eqref{eq:equiv_t>3} is
	\[ b^q+b^{q^{t-1}}-h^{1-q^t}b^{q^{t+1}} - h^{1-q^t}b^{q^{2t-1}} \]
	which must be $0$. Plugging \eqref{eq:equiv_t>3_b1} into it and taking into account that $h^{q^t}=-1/h$, we get
	\[b^q(1+h^{q^2-q^t}) + b^{q^{t-1}}(1+h^{1-q^{t-2}})=0.\]
	Raising it to the $q^2$-th power and plugging \eqref{eq:equiv_t>3_b1} again, we obtain
	\[ b^{q^3} (1-h^{q^4+q^2}) - b^q(h^{q^2-1} - h^{2q^2})=0,\]
	which means 
	\[b^q(h^{-1}-h^{q^2}) =\left(b^q(h^{-1}-h^{q^2}) \right)^{q^2}.\]
	 By Proposition \ref{p:hcondition}, $h^{-1}\neq h^{q^2}$. Thus 
	\[b=\frac{-\delta}{h^{-q^{2t-1}}-h^q}=\frac{\delta}{h^{q^{t-1}}+h^q}\]
	for some $\delta\in \F_{q^2}$. Substitute it into \eqref{eq:equiv_t>3_b1},
	\[\frac{\delta^{q^t}}{h^{q^{2t-1}}+h^{q^{t+1}}} - h^{q+q^{t-1}}\frac{\delta}{h^{q^{t-1}}+h^q}=0,\]
	which means
		\[\delta^{q^t}+\delta=0.\]
	When $t$ is even, it means $2\delta=0$. Thus $b=0$.
	
	When $t$ is odd, it implies $\delta^{q}=-\delta$. Plugging this value back into \eqref{eq:equiv_t>3}, we see that the coefficients of $x^{q^2}$, $x^{q^{t-2}}$, $x^{q^t}$,  $x^{q^{t+2}}$ and $x^{q^{2t-2}}$ on the right-hand-side of \eqref{eq:equiv_t>3} are all $0$. Hence, the coefficient of $x$ on the left-hand-side of \eqref{eq:equiv_t>3} is $c$ which is completely determined by $b$, and $a$ can take any value in $\F_q$ when $b\neq 0$. Moreover, the value of $d$ is determined by $a$ which is independent of the value of $b$. One may check directly by computation that they form a cyclic group of order $q^2-1$, or use the fact that the nonzero elements of the right idealizer of an MRD code form the multiplicative group of a finite field (the link between $I_R(\cC_{h,t})$ and this group is already given in Lemma \ref{le:full_auto_MRD}).
\end{proof}

Theorem \ref{t:code_equiv} shows that our construction provides a big family of inequivalent MRD codes.
\begin{corollary}\label{coro:BIG}
	Let $p$ be an odd prime number and let $r,t$ be positive integers with $t>4$ and $q=p^r$. The total number $N$ of inequivalent MRD codes $\cC_{h,t}$ is

		\[ N\geq \begin{cases}
		\left\lfloor\frac{q^t+1}{4rt}\right\rfloor,& \text{ if }t\not\equiv 2\pmod{4};\\[10pt]
		\left\lfloor\frac{q^t+1}{2rt(q^2+1)}\right\rfloor,& \text{ if }t\equiv 2\pmod{4}.
		\end{cases} \]	

\end{corollary}
\begin{proof}
	In the proof of Theorem \ref{t:code_equiv}, we have obtained a necessary and sufficient condition \eqref{eq:GL(2,q^n)-equi} for the $\GL(2,q^n)$-equivalence between $U_{h}$ and $U_{k}$, $n=2t$. 	For a given $h$ satisfying $h^{q^t+1}=-1$, let $\xi_h$ denote the number of $k$ for which $U_k$ is $\GL(2,q^n)$-equivalent to $U_h$. The value of $\xi_h$ is independent of $h$ and
	\[\xi_h=\begin{cases}
		2, & t\not\equiv 2\pmod{4};\\
		q^2+1, & t\equiv 2\pmod{4}.
	\end{cases}\]
	As there are at most $|\Aut(\F_{q^n})|=rn$ different $\rho$ such that $U_h$ is $\GL(2,q^n)$-equivalent to $U_{k^\sigma}$ for a given $h$,  there are at most $nr\xi_h$ choices of $k$ for which $U_{h}$ is $\GaL(2,q^n)$-equivalent to $U_{k}$. Therefore, we have obtained the lower bound for $N$ which concludes the proof.		
\end{proof}

By Corollary \ref{coro:BIG}, we can prove the following result.

\begin{theorem}
	Let $n=2t$ with $t>4$ and let $q$ be an odd prime power. The family of $\F_{q^n}$-MRD codes of minimum distance $n-1$
	$$\cC_{h,t}=\{ ax+b\psi_{h,t}(x): a,b\in \F_{q^n} \},$$
	where $\psi_{h,t}(x)=  x^{q}+x^{q^{t-1}}-h^{1-q^{t+1}}x^{q^{t+1}}+h^{1-q^{2t-1}}x^{q^{2t-1}}\in\F_{q^n}[x]$ and $h$ is any element of $\F_{q^n}$ such that $h^{q^t+1}=-1$, is new.
\end{theorem}
\begin{proof}
	As  $q$ can be any odd prime power and $n$ can be any even integer larger than $6$, by comparing Corollary \ref{coro:BIG} with the numbers of known inequivalent constructions of $\F_{q^n}$-MRD codes of minimum distance $n-1$ in Table \ref{table:numbers}, our family must be new.
\end{proof}

The following result is a direct consequence of Theorem \ref{t:code_equiv} and Lemma \ref{le:full_auto_MRD}.
\begin{corollary}\label{coro:idealizer}
	Let $\cC_{h,t}$ be defined as in \eqref{eq:C_ht}, with $t>4$. Then $I_R(\cC_{h,t})\cong\F_{q^{2}}$.
\end{corollary}

Finally, let us investigate the adjoint of $\psi_{h,t}$ and the associated MRD code.
\begin{theorem}\label{t:adjoint}
	The $\F_{q^n}$-MRD $\hat{\cC}_{h,t}=\{ ax+b\hat{\psi}_{h,t}(x): a,b\in \F_{q^n} \}$, where $\hat{\psi}_{h,t}$ is the adjoint map of $\psi_{h,t}$ is equivalent to $\cC_{h,t}$.
\end{theorem}
\begin{proof}
	Consider the polynomial
	\[g(x):=h\hat{\psi}_{h,t}(x/h) =x^q - x^{q^{t-1}} + h^{1-q^{t+1}} x^{q^{t+1}}+ h^{1-q^{2t-1}}x^{q^{2t-1}},\]
	we investigate the equivalence between $\cC_{g}$ and $\cC_{h,t}$.
	
	To prove that ${\cC}_{g}$ is equivalent to $\cC_{h,t}$,  as in the proof of Theorem \ref{t:code_equiv}, we only have to consider the existence of invertible matrix 
	$\begin{pmatrix}
	a&b\\c&d
	\end{pmatrix}$ over $\F_{q^n}$ such that for each $x\in \F_{q^n}$ there exists $y\in \F_{q^n}$ satisfying
	\[
	\begin{pmatrix}
	a&b\\c&d
	\end{pmatrix}
	\begin{pmatrix}
	x\\ g(x)
	\end{pmatrix}=
	\begin{pmatrix}
	y\\\psi_{h,t}(y)
	\end{pmatrix}.
	\]
	This is equivalent to
	\begin{equation}\label{eq:equiv_t>3_adjoint}
	cx+d g(x)=\psi_{h,t}\left(ax+bg(x)\right),
	\end{equation}
	for all $x\in \F_{q^n}$.
	
	Let $a=d=0$. We only have to check whether there exist nonzero $b$ and $c$ such that 
	\[ cx = \psi_{h,t}(bg(x)).\]
	The right-hand-side of it is
	\begin{align*}
	&\left(b^q - h^{1-q^{2}}b^{q^{t+1}}\right)x^{q^2} + \left(b^{q^{t-1}}h^{q^{t-1}-q^{t-2}} - b^{q^{2t-1}}h^{1-q^{2t-1}}\right)x^{q^{t-2}} +\\
	+&\left(-b^q+b^{q^{t-1}}-h^{1-q^t}b^{q^{t+1}}  + h^{1-q^{t}}b^{q^{2t-1}}\right)x^{q^t}+\\
	+&\left(- b^{q^{t-1}}+h^{1-q^{2t-2}}b^{q^{2t-1}} \right)x^{q^{2t-2}} + \left(h^{q-q^{t+2}}b^q-h^{1-q^{t+1}}b^{q^{t+1}}\right)x^{q^{t+2}}+\\
	+&\left(b^qh^{q-1} + b^{q^{t-1}}h^{q^{t-1}-1}  + h^{1-q^{t+1}}b^{q^{t+1}}+h^{1-q^{2t-1}}b^{q^{2t-1}}\right)x.
	\end{align*}
	
	Choose $z\in\F_q$ and let $b=\left(\frac{zh}{h^{q^2+1}-1}\right)^{q^{2t-1}}$. Since $z^{q^t-1}=1$, we get
	\[
	(b^q)^{q^{t}-1} = \frac{z^{q^t-1} h^{q^t-1}(h^{q^2+1}-1)}{h^{q^{t+2}+q^t}-1}= z^{q^t-1} h^{q^t+q^2}= h^{q^2-1},
	\]	
	which means $b^{q^{t+1}} = h^{q^2-1}b^q$.
	
	It is easy to verify that  the coefficients of the terms of degree $q^2$, $q^{t-2}$, $q^{2t-2}$ and $q^{t+2}$ in $\psi_{h,t}(bg(x))$ all equal $0$. 
	
	The computation of the coefficient of $x^{q^t}$ is a bit more complicated. It equals
	\begin{align*}
	&-b^q+b^{q^{t-1}}-h^{1-q^t}b^{q^{t+1}}  + h^{1-q^{t}}b^{q^{2t-1}}=\\
	=&-b^q+b^{q^{t-1}}+h^{2}b^{q^{t+1}}  - h^{2}b^{q^{2t-1}}=\\
	=&-b^q+b^{q^{t-1}}+h^{q^2+1}b^{q}  - h^{q^{2t-2}+1}b^{q^{t-1}}=\\
	=& b^q(h^{q^2+1}-1) - b^{q^{t-1}}(h^{-q^{2}-1}-1)^{q^{t-2}}=\\
	=& b^q(h^{q^2+1}-1) - b^{q^{t-1}}\left(h^{q^{2}+1}-1\right)^{q^{t-2}} h^{1-q^{t-2}}.
	\end{align*}
	It equals $0$ if and only if 
	\[ \frac{b^q}{h}(h^{q^2+1}-1) = \left( \frac{b^q}{h}(h^{q^2+1}-1) \right)^{q^{t-2}}, \]
	which holds because $\frac{b^q}{h}(h^{q^2+1}-1)=z \in \F_q\subseteq\F_{q^{t-2}}\cap \F_{q^{2t}}$. 
	
	Therefore, we have shown that there exist $c$ and $b$ such that $cx=\psi_{h,t}(bg(x))$, which means that $\cC_g$ and $\cC_{k,t}$ are equivalent.
\end{proof}

\section{Equivalence of the associated linear sets}\label{sec:PGaL_equivalence}
Let $\psi_{h,t}$ be the scattered polynomial over $\F_{q^n}$, $n=2t$, defined in Theorem \ref{t:main}. Let 
\begin{equation}\label{familyofL}
    L_{h,t} := \left\{ \langle (x,\psi_{h,t}(x)) \rangle_{\F_{q^n}}: x\in \F^*_{q^n}\right\},
\end{equation}
which is a maximum scattered linear set of $\PG(1,q^n)$.

In this part,  we consider the $\PGaL$-equivalence between $L_{h,t}$ and $L_{k,t}$. Our main result shows that there is a large number of inequivalent maximum scattered linear sets associated with our family of scattered linear sets.
\begin{theorem}\label{t:PGaL-equivalence}
	Let $p$ be an odd prime number and let $r,t$ be positive integers with $t>4$ and $q=p^r$. The total number $M$ of inequivalent maximum scattered linear sets $L_{h,t}$ of $\PG(1,q^n)$, $n=2t$, satisfies
	\[ M\geq\begin{cases}
	\left\lfloor\frac{q^t+1}{8rt}\right\rfloor,& \text{ if }t\not\equiv 2\pmod{4};\\[10pt]
	\left\lfloor\frac{q^t+1}{4rt(q^2+1)}\right\rfloor,& \text{ if }t\equiv 2\pmod{4}.
	\end{cases} \]
\end{theorem}

To prove Theorem \ref{t:PGaL-equivalence}, we first restrict to the equivalence of linear sets under $\PGL(2,q^n)$ and consider two cases which will be handled in Lemma \ref{le:LS_f-g_equiv_case1} and Lemma \ref{le:LS_f-g_equiv_case2}, respectively. Then we will consider the $\PGaL(2,q^n)$-equivalence and present the proof of Theorem \ref{t:PGaL-equivalence}.

\medskip

Let $f(x)=\sum_{i=0}^{n-1} \alpha_i x^{q^i}$ and $g(x)=\sum_{i=0}^{n-1} \beta_i x^{q^i}$ be two scattered polynomials over $\F_{q^n}$ with $\alpha_0=\beta_0=0$. The associated linear sets $L_f$ and $L_g$ are $\PGL(2,q^n)$-equivalent if and only if there exists an invertible matrix 
$\begin{pmatrix}
	a & b\\
	c & d
\end{pmatrix}$ over $\F_{q^n}$ such that
\begin{equation}\label{eq:LS_f-g_equiv}
	 \left\{\frac{f(x)}{x}:x\in \F_{q^n}^*  \right\} =\left\{\frac{cx+dg(x)}{ax+bg(x)}:x\in \F_{q^n}^*  \right\}.
\end{equation} 
Depending on whether the value of $b$ equals $0$ or not, we may consider the equivalence in two cases.

If $b=0$, then we can assume that $a=1$, and \eqref{eq:LS_f-g_equiv} becomes
\[ \left\{  \sum_{i=0}^{n-1} \alpha_i x^{q^i-1}:x\in \F_{q^n}^* \right\}  = \left\{c+d\sum_{i=0}^{n-1} \beta_i x^{q^i-1} :x\in \F_{q^n}^*  \right\}. \]

As $\alpha_0=\beta_0=0$, by Lemma 3.6 in \cite{csajbok_classes_2018}, $c$ must be $0$. Hence
\begin{equation}\label{eq:LS_f-g_equiv_case1}
	\left\{  \sum_{i=1}^{n-1} \alpha_i x^{q^i-1}:x\in \F_{q^n}^* \right\}  = \left\{d\sum_{i=1}^{n-1} \beta_i x^{q^i-1} :x\in \F_{q^n}^*  \right\}.
\end{equation}
\begin{lemma}\label{le:LS_f-g_equiv_case1}
	Suppose $f=\psi_{h,t}$, $g=\psi_{k,t}$ with $t > 4$. If there exists $d$ such that \eqref{eq:LS_f-g_equiv_case1} holds, then $(h/k)^{q^2+1}=1$ which means $U_h$ is $\GaL(2,q^n)$-equivalent to $U_k$.
\end{lemma}
\begin{proof}
	By Lemma \ref{le:jcta2018},
	\[\alpha_j\alpha_{n-j}^{q^j} = d^{q^j+1}\beta_k\beta_{n-j}^{q^j}\]
	for $j=1,2,\cdots, n-1$. Plugging the coefficients of $\psi_{h,t}$ and $\psi_{k,t}$, we get
	\[ (h^{1-q^{2t-1}})^q=d^{q+1} (k^{1-q^{2t-1}})^q \text{ and }  (-h^{1-q^{t+1}})^{q^{t-1}}=d^{q^{t-1}+1} (-k^{1-q^{t+1}})^{q^{t-1}}.\]
	By setting $\ell= h/k$, we get
	\begin{equation}\label{eq:lb_and_lb}
		\begin{cases}
			\ell^{q-1}=d^{q+1},\\
			\ell^{q^{t-1}-1} = d^{q^{t-1}+1}.
		\end{cases}
	\end{equation}
	From \eqref{eq:lb_and_lb}, we derive 
	\[ d^{(q+1)(q^{t-2}+q^{t-3}+\cdots+1)} = \ell^{q^{t-1}-1 }= d^{q^{t-1}+1},\]
	which means $d^{2q(q^{t-3}+\cdots+1)}=1$. Hence $d^{q^{t-2}-1}= d^{2(q^{t-3}+\cdots+1)\cdot \frac{q-1}{2}}=1$. It follows that $\ell^{q^{t-1}-1} =d^{q^{t-1}+1}=d^{q+1}=\ell^{q-1}$, whence $\ell^{q^{t-2}}=\ell$. 
	
	As $h^{q^t+1}=k^{q^t+1}=-1$, $\ell^{q^t}=1/\ell$. By $\ell^{q^{t-2}}=\ell$, we get $\ell^{q^2+1}=1$. By Theorem \ref{t:code_equiv}, $U_h$ is $\GaL(2,q^n)$-equivalent to $U_k$.
\end{proof}

Next we consider the case $b\neq 0$. Without loss of generality, we assume that $b=1$.  Then
\[ \frac{cx+dg(x)}{ax+g(x)} = \frac{(c-da)x+d(ax+g(x))}{ax+g(x)}=\frac{\bar{c}x}{ax+g(x)}+d,\]
for any $x\in \F_{q^n}^*$, where $\bar{c}=c-da\neq 0$. Noting that $ax+g(x)=0$ must have no nonzero solution, \[\frac{\bar{c}x}{ax+g(x)}+d=\frac{\bar{c}\bar{g}(y)}{y}+d, \]
where $\bar{g}(y)=\sum_{i=0}^{n-1}\gamma_i y^{q^i}$ is the inverse of the map $x\mapsto ax+g(x)$. 

Furthermore, \eqref{eq:LS_f-g_equiv} becomes
\begin{equation*}
		\left\{  \sum_{i=1}^{n-1} \alpha_i x^{q^i-1}:x\in \F_{q^n}^* \right\}  = \left\{\bar{c}\sum_{i=1}^{n-1}\gamma_i x^{q^i-1}+d+\bar{c}\gamma_0 :x\in \F_{q^n}^*  \right\}.
\end{equation*}
By Lemma \ref{le:jcta2018}, 
\begin{equation}\label{eq:d+cgamma_0=0}
	d+\bar{c}\gamma_0=0.	
\end{equation}
Thus
\begin{equation}\label{eq:LS_f-g_equiv_case2}
	\left\{  \frac{1}{\bar{c}} \sum_{i=1}^{n-1} \alpha_i x^{q^i-1}:x\in \F_{q^n}^* \right\}  = \left\{\sum_{i=1}^{n-1}\gamma_i x^{q^i-1} :x\in \F_{q^n}^*  \right\}.
\end{equation}

\begin{lemma}\label{le:LS_f-g_equiv_case2}
	Let $f=\psi_{h,t}$, $g=\psi_{k,t}$ with $t >4$. Suppose that  there exist $a$, $c$ and $d$ such that \eqref{eq:LS_f-g_equiv_case2} holds.
	\begin{enumerate}[label=(\alph*)]
		\item If $t$ is even, then $a$ must be $0$.
		\item If $t$ is odd and $a\neq 0$, then $f=g$.
	\end{enumerate}
	In particular, when $a=0$, $\gamma_0=d=0$.
\end{lemma}
\begin{proof}
	By the second identity in Lemma \ref{le:jcta2018}, we know that 
	\begin{equation}\label{eq:second_eq}
		\gamma_j \gamma_{2t-j}=0,
	\end{equation}
	 for $j\in \{1,2,\cdots, 2t-1\} \setminus \{1, t-1, t+1, 2t-1\}$, 
	 \begin{equation}\label{eq:second_eq_1}
	 	\gamma_1\gamma_{2t-1}^q= \left( \frac{1}{\bar{c}} \right)^{q+1}\tau^q,
	\end{equation}
	and
	\begin{equation}\label{eq:second_eq_t-1}
		\gamma_{t-1} \gamma_{t+1}^{q^{t-1}}= \left( \frac{1}{\bar{c}} \right)^{1+q^{t-1}}\theta^{q^{t-1}},
	\end{equation}
	where $\tau = h^{1-q^{2t-1}}$ and $\theta=-h^{1-q^{t+1}}=h^{1+q}$.

	 By the third identity in Lemma \ref{le:jcta2018}, we have
	 \begin{equation}\label{eq:third_eq}
	 	\gamma_1 \gamma_{j-1}^q \gamma_{2t-j}^{q^j} + \gamma_{j} \gamma_{2t-1}^q \gamma_{2t-j+1}^{q^j}=0,
	 \end{equation}
	for $j\in \{2,3,\cdots, n-1\}$.
	Letting $j=2$, we obtain
	\[\gamma_1 \gamma_1^q \gamma_{2t-2}^{q^2} + \gamma_2 \gamma_{2t-1}^q \gamma_{2t-1}^{q^2}=0. \]
	As $\gamma_2\gamma_{2t-2}=0$ and $\gamma_1\gamma_{2t-1}\neq 0$, we derive $\gamma_2=\gamma_{2t-2}=0$. Similarly, by letting $j=t-1$, we have
	\[\gamma_1 \gamma_{t-2}^q \gamma_{t+1}^{q^{t-1}} + \gamma_{t-1} \gamma_{2t-1}^q \gamma_{t+2}^{q^{t-1}}=0. \]
	Since $\gamma_{t-2}\gamma_{t+2}=0$ and $\gamma_1,\gamma_{t+1}, \gamma_{t-1}, \gamma_{t+1}\neq 0$, from the above equation we deduce 
	$$\gamma_{t-2}=\gamma_{t+2}=0.$$
	Moreover, by \eqref{eq:second_eq}, \eqref{eq:third_eq} and replacing $j-1$ by $j$ in  \eqref{eq:third_eq}, 
	\begin{equation}\label{eq:gamma_jneq0}
		\gamma_j \neq 0 \Rightarrow \gamma_{2t-j}=\gamma_{2t-j+1}=\gamma_{2t-j-1}=0,
	\end{equation}
	for $j\in \{1,2,\cdots, 2t-1\} \setminus \{1, t-1, t+1, 2t-1\}$.
	
	Now, we will use the fact that $\bar{g}(ax+g(x))=x$, namely,
	\begin{equation}\label{eq:bar_g}
		 \bar{g}(ax+ x^q + x^{q^{t-1}}+u x^{q^{t+1}} +vx^{q^{2t-1}} )=x, 
	\end{equation}
	for all $x\in \F_{q^n}$, where $u=-k^{1-q^{t+1}}$ and $v=k^{1-q^{2t-1}}$. 	The coefficient of $x^{q^j}$ in the left-hand-side of \eqref{eq:bar_g} is
	\begin{equation}\label{eq:LS_f-g_equiv_case2_x^q^j}
		a^{q^j} \gamma_j+\gamma_{j-1} +\gamma_{j+t+1} + u^{q^{j+t-1}}\gamma_{j+t-1} + v^{q^{j+1}}\gamma_{j+1}. 
	\end{equation}	
	
	By letting $j=0, 1$, $2t-1$, $t+1$, $t-1$ and $t$ in \eqref{eq:LS_f-g_equiv_case2_x^q^j} and comparing it with the right-hand-side of \eqref{eq:bar_g}, we get
	\begin{align}
		\label{eq:gamma_j=0} a\gamma_0 + \gamma_{2t-1} + \gamma_{t+1} + u^{q^{t-1}}\gamma_{t-1} + v^q\gamma_{1} &=1,\\
		\label{eq:gamma_j=1}  a^q\gamma_1 + \gamma_0 &= 0,\\
		\label{eq:gamma_j=2t-1}		a^{q^{2t-1}}\gamma_{2t-1} + v\gamma_0 &= 0,\\
		\label{eq:gamma_j=t+1}  a^{q^{t+1}}\gamma_{t+1} +u\gamma_0 &= 0,\\
		\label{eq:gamma_j=t-1} a^{q^{t-1}} \gamma_{t-1} +\gamma_0 &= 0,\\
		\label{eq:gamma_j=t}  \gamma_{t-1} +\gamma_1 +u^{q^{2t-1}}\gamma_{2t-1}+v^{q^{t+1}}\gamma_{t+1} &=0.
	\end{align}
	Here we have used the result that $\gamma_t=\gamma_2=\gamma_{2t-2}=\gamma_{t+2}=\gamma_{t-2}=0$.
	
	It is clear that if $a=0$, then $\gamma_0$ must be $0$.  By \eqref{eq:d+cgamma_0=0}, $d=0$.
	
	Assume that $a\neq 0$. By \eqref{eq:gamma_j=1}, \eqref{eq:gamma_j=2t-1}, \eqref{eq:gamma_j=t+1} and \eqref{eq:gamma_j=t-1} into \eqref{eq:gamma_j=0} and \eqref{eq:gamma_j=t}, we see that $\gamma_0\neq 0$ is completely determined by $a$, and 
	\begin{equation*}
		-\gamma_0\left( \frac{1}{a^{q^{t-1}}} + \frac{1}{a^{q}} + \frac{u^{q^{2t-1}}v}{a^{q^{2t-1}}}   + \frac{v^{q^{t+1}}u}{a^{q^{t+1}}} \right)=0, 
	\end{equation*} respectively.
	Recall that $u=-k^{1-q^{t+1}}$, $v=k^{1-q^{2t-1}}$ and $k^{q^t}=-1/k$, from the above equation we deduce
	\[\frac{1}{a^{q^{t-1}}} + \frac{1}{a^{q}} + \frac{k^2}{a^{q^{2t-1}}}   + \frac{k^2}{a^{q^{t+1}}}=0. \]
	Therefore
	\begin{equation}\label{eq:gamma_j=t-further}
		 \frac{1}{a^{q^{t-1}}} + \frac{1}{a^{q}} = - k^2\left(\frac{1}{a^{q^{t-1}}}   + \frac{1}{a^{q}}\right)^{q^t}. 
	\end{equation}
	
	Our goal of the next step is to prove
	\begin{equation}\label{eq:a's_power_qt}
		\frac{1}{a^{q^t}} = -k^{q-q^{2t-1}} \frac{1}{a},
	\end{equation}
	always holds.
	
	Let $j=2, t+2, 2t-2, t-2$ in \eqref{eq:LS_f-g_equiv_case2_x^q^j}, we have
	\begin{align}
		\label{eq:gamma_j=2} \gamma_1 + \gamma_{t+3} + u^{q^{t+1}} \gamma_{t+1} + v^{q^3} \gamma_3 &= 0,\\
		\label{eq:gamma_j=t+2} \gamma_{t+1} + \gamma_{3} + u^{q} \gamma_{1} + v^{q^{t+3}} \gamma_{t+3} &= 0,\\
		\label{eq:gamma_j=2t-2} \gamma_{2t-3} + \gamma_{t-1} + u^{q^{t-3}} \gamma_{t-3} + v^{q^{2t-1}} \gamma_{2t-1} &= 0,\\
		\label{eq:gamma_j=t-2} \gamma_{t-3} + \gamma_{2t-1} + u^{q^{2t-3}} \gamma_{2t-3} + v^{q^{t-1}} \gamma_{t-1} &= 0.
	\end{align}
	
	Depending on the value of $\gamma_3$,  $\gamma_{t+3}$,  $\gamma_{t-3}$ and  $\gamma_{2t-3}$, we separate the  proof of \eqref{eq:a's_power_qt} into four different cases.
	
	\textbf{Case (i)}. $\gamma_3=\gamma_{t+3}=0$. By \eqref{eq:gamma_j=1}, \eqref{eq:gamma_j=t+1} and \eqref{eq:gamma_j=2},
	\[ -\frac{u}{a^{q^{t+1}}} \gamma_0=\gamma_{t+1}=-u^q\gamma_1 =\frac{u^q}{a^q} \gamma_0, \]
	which means $\frac{1}{a^{q^t}} = -u^{1-q^{2t-1}} \frac{1}{a} = -k^{q-q^{2t-1}} \frac{1}{a}$.
	
	\textbf{Case (ii)}. $\gamma_{t-3}=\gamma_{2t-3}=0$. By a similar computation of \eqref{eq:gamma_j=2t-2} as in Case (i), we get  \eqref{eq:a's_power_qt} again.
	
	\textbf{Case (iii)}.	$\gamma_3\neq 0$ and $\gamma_{t-3}\neq  0$.  By \eqref{eq:gamma_jneq0}, $\gamma_{2t-3}=\gamma_{2t-4}=\gamma_{t+3}=\gamma_{t+4}=0$. Now, \eqref{eq:gamma_j=2} and \eqref{eq:gamma_j=t+2} become
	\begin{align*}
		\gamma_1 + k^{-q^2-q}\gamma_{t+1} + k^{q^3-q^2} \gamma_3 &=0,\\
		\gamma_3 + \gamma_{t+1} + k^{q^2+q} \gamma_1 &=0.
	\end{align*}

	Canceling $\gamma_3$, we get
	\[(1-k^{q^3+q}) \gamma_{t+1} = - (1-k^{q^3+q})k^{q^2+q}\gamma_1. \]
	By Proposition \ref{p:hcondition}, $k^{q^2+1}\neq 1$. Hence, $\gamma_{t+1} = - k^{q^2+q}\gamma_1$. By plugging \eqref{eq:gamma_j=1} and \eqref{eq:gamma_j=t+1}  into it, we derive \eqref{eq:a's_power_qt}.
	
	\textbf{Case (iv)}.	$\gamma_{t+3}\neq 0$ and $\gamma_{2t-3}\neq  0$. By \eqref{eq:gamma_jneq0}, $\gamma_{t-3}=\gamma_{t-4}=\gamma_{3}=\gamma_{4}=0$. As in Case (iii), by canceling $\gamma_{t+3}$ using \eqref{eq:gamma_j=2} and  \eqref{eq:gamma_j=t+2}, we obtain \eqref{eq:a's_power_qt} again.
	
	By \eqref{eq:gamma_jneq0}, we have covered all possible cases. Therefore,  \eqref{eq:a's_power_qt} is proved.
	
	Now we are ready to prove (a) and (b).  Our strategy is to give a precise expression for $a$, which is strong enough to prove (a). Then we further use \eqref{eq:second_eq_1} to get more restrictions on the value of $h$ which leads to (b).
	
	Plugging \eqref{eq:a's_power_qt} in \eqref{eq:gamma_j=t-further}, we have
	\[   \frac{1}{a^{q^{t-1}}} \left(1-k^{1+q^{2t-2}}\right)  + \frac{1}{a^q}\left(1-k^{1+q^2}\right)  =0,\] 
	that is 
	\[   \left( \frac{1}{a^q}\left(1-k^{1+q^2}\right) \right)^{q^{t-2}} -k^{-1-q^{2t-2}} \frac{1}{a^q}\left(1-k^{1+q^2}\right)=0.\] 
	By $k^{q^{t-2}} = -k^{-q^{2t-2}}$, we have
	\[ \left( \frac{k^{-1}}{a^q}\left(1-k^{1+q^2}\right) \right)^{q^{t-2}} +\frac{k^{-1}}{a^q}\left(1-k^{1+q^2}\right)=0. \]
	Therefore, 
	\begin{equation}\label{eq:c_value}
		a=k^{-q^{2t-1}} \left(1-k^{q+q^{2t-1}}\right)\eta, 
	\end{equation}
	where $\eta$ satisfies $\eta^{q^{t-2}}+\eta = 0$. Since $0=(\eta^{q^{t-2}}+\eta)^{q^{t+2}} = \eta+(\eta^{q^{t-2}})^{q^4}=\eta-\eta^{q^4}$, we get $\eta\in \F_{q^4}$. Moreover,
	\[
	\eta = -\eta^{q^{t-2}}=
	\begin{cases}
		-\eta^{q^3}=-\eta^q, & t\equiv 1\pmod{4},\\
		-\eta=0, & t\equiv 2\pmod{4},\\
		-\eta^q, & t\equiv 3\pmod{4},\\
		-\eta^{q^2}, & t\equiv 0\pmod{4}.
	\end{cases}
	\]
	Hence
	\begin{equation}\label{eq:eta^q2-1}
		\eta^{q^2-1}=
		\begin{cases}
		1, & t\equiv 1\pmod{4},\\
		0, & t\equiv 2\pmod{4},\\
		1, & t\equiv 3\pmod{4},\\
		-1, & t\equiv 0\pmod{4}.
		\end{cases}
	\end{equation}
	
	Substitute $a$ in \eqref{eq:a's_power_qt} by \eqref{eq:c_value}, 
	\[
	-k^{q-q^{2t-1}}k^{-q^{t-1}}(1-k^{q^{t+1}+q^{t-1}}) \eta^{q^{t}} = k^{-q^{2t-1}}(1-k^{q+q^{2t-1}})\eta,
	\]
	which equals
	\[
		k^{q}(1-k^{-q-q^{2t-1}}) \eta^{q^{t}} = k^{-q^{2t-1}}(1-k^{q+q^{2t-1}})\eta.
	\]
	It implies 
	\begin{equation}\label{eq:eta}
		\eta^{q^t} = -\eta.
	\end{equation} Together with $\eta^{q^{t-2}}+\eta = 0$, we deduce $\eta^{q^2} =\eta$. It contradicts \eqref{eq:eta^q2-1} when $t\equiv 0\pmod{4}$. Therefore, when $t$ is even, $a$ must be $0$ and (a) is proved.

	Finally, let us plugging \eqref{eq:gamma_j=1} and \eqref{eq:gamma_j=2t-1} into \eqref{eq:second_eq_1}, we have 
	\begin{equation}\label{eq:gamma_1_bara}
	( \bar{c} \gamma_0)^{q+1}=\frac{\tau^q}{v^q} a^{q+1}=\ell^{q-1}a^{q+1},
	\end{equation}
	where $\ell=h/k$ which means $\ell^{q^t}=1/\ell$.
	
	Similarly, by \eqref{eq:gamma_j=t-1} and \eqref{eq:gamma_j=t+1} into \eqref{eq:second_eq_t-1}
	\begin{equation}\label{eq:gamma_t-1_bara}
	( \bar{c} \gamma_0)^{q^{t-1}+1}=\frac{\theta^{q^{t-1}}}{u^{q^t-1}} a^{q^{t-1}+1}=\ell^{q^{t-1}-1}a^{q^{t-1}+1}.
	\end{equation}
	Raising \eqref{eq:gamma_1_bara} to its $\frac{q^{t-1}+1}{2}$-th power and  \eqref{eq:gamma_t-1_bara} to its $\frac{q+1}{2}$-th power and canceling $(\bar{c}\gamma_0)^{\frac{1}{2}(q+1)(q^{t-1}+1)}$ and $a^{\frac{1}{2}(q+1)(q^{t-1}+1)}$, we obtain
	\[ \ell^{q^{t-2}-1}=1.   \]
	Again, by $\ell^{q^t+1}=1$, we have $\ell^{q^2+1}=1$ which means $\ell^{\gcd(q^2+1, q^t+1)}=1$. Therefore, since $t$ is odd, we get $\ell^2=1$, and (b) is proved.	
\end{proof}

Now we are ready to prove Theorem \ref{t:PGaL-equivalence}.
\begin{proof}[Proof of Theorem \ref{t:PGaL-equivalence}]
	Suppose that $f=\psi_{h,t}$, $g=\psi_{k,t}$, and $L_f$ is $\PGaL(2,q^n)$-equivalent to $L_g$. Then there exists an invertible matrix 
	$\begin{pmatrix}
	a & b\\
	c & d
	\end{pmatrix}$ over $\F_{q^n}$ and $\sigma\in \Aut(\F_{q^n})$ such that
	\begin{equation*}
	\left\{\frac{f(x)}{x}:x\in \F_{q^n}^*  \right\} =\left\{\frac{cx^\sigma+d(g(x))^\sigma}{ax^\sigma+b(g(x))^\sigma}:x\in \F_{q^n}^*  \right\} =\left\{\frac{cx+d\bar{g}(x)}{ax+b\bar{g}(x)}:x\in \F_{q^n}^*  \right\},
	\end{equation*} 
	where $\bar{g}(x)=\psi_{k^\sigma, t}(x)$.
	
	For a given $h$ satisfying $h^{q^t+1}=-1$, let $\varepsilon_h$ denote the number of $k$ for which $U_k$ is $\GaL(2,q^n)$-equivalent to $U_h$ .
	
	Depending on the value of $b$, we separate the remainder part of the proof into two cases:
	
	\textbf{Case (a).} $b=0$. By Lemma \ref{le:LS_f-g_equiv_case1}, $\left(h/k^\sigma\right)^{q^2+1}=1$ which means that $U_h$ is $\GaL(2,q^n)$-equivalent to $U_{k^\sigma}$. Thus there are exactly $\varepsilon_h$ choices of $k$ for which $L_{k,t}$ is equivalent to $L_{h,t}$.
	
	\textbf{Case (b).} $b\neq 0$. Without loss of generality, we assume $b=1$ and $f\neq \bar{g}$. By Lemma \ref{le:LS_f-g_equiv_case2}, $d=a=0$. Thus
	\[
		\left\{\frac{f(x)}{x}:x\in \F_{q^n}^*  \right\}  =\left\{\frac{cx}{\bar{g}(x)}:x\in \F_{q^n}^*  \right\}.
	\]
	Suppose that there is another $\tilde{g}(x)=\psi_{\tilde{k},t}(x)$ for certain $\tilde{k}\in \F_{q^{2t}}$ satisfying $\tilde{k}^{q^t+1}=-1$ and 
	\[
		\left\{\frac{f(x)}{x}:x\in \F_{q^n}^*  \right\}  =\left\{\frac{\tilde{c}x}{\tilde{g}(x)}:x\in \F_{q^n}^*  \right\}
	\]
	for some $\tilde{c}\in \F_{q^{2t}}$. Then
		\[
	\left\{ \frac{\bar{g}(x)}{x}:x\in \F_{q^n}^*  \right\}  =\left\{ \frac{c}{\tilde{c}} \cdot \frac{\tilde{g}(x)}{x} :x\in \F_{q^n}^*  \right\}.
	\]
	By Lemma \ref{le:LS_f-g_equiv_case1}, $(\tilde{k}/k^\sigma)^{q^2+1}=1$ which means that $U_{k^\sigma}$ is $\GaL(2,q^n)$-equivalent to $U_{\tilde{k}}$.  
	Hence, for the case $b\neq0$, there are exactly $\varepsilon_{\tilde{k}}$ choices of $k$ for which $L_{k,t}$ is equivalent to $L_{h,t}$.
	
	Finally we combine Case (a) and Case (b), for a given $L_{h,t}$. Noting that $|\{ h\in \F_{q^n}: h^{q^t+1}=-1 \} |=q^t+1$, there are exactly 
		\[M=\frac{q^t+1}{\varepsilon_h}+\frac{q^t+1}{\varepsilon_{\tilde{k}}}
		\]
	inequivalent $L_{h,t}$ defined by \eqref{familyofL}.  Recall that 
	\[\varepsilon_h\leq n r\xi_h=
		\begin{cases}
			4r t, & t\not\equiv 2\pmod{4};\\
			 2(q^2+1)rt , &t\equiv 2\pmod{4},
		\end{cases}
		\]
		for every possible choice of $h$; see the proof of Corollary \ref{coro:BIG}. We obtain the lower bound of $M$.
\end{proof}
\begin{remark}
	By Theorem \ref{t:PGaL-equivalence}, Family \eqref{familyofL} contains much more inequivalent elements compared with the known constructions for infinitely many $n$ listed in Table \ref{table:numbers}. Therefore, this family must be new.
\end{remark}
	\begin{remark}
		For $t=3$ and $4$, we do not have any result about the equivalence problems among the linear sets defined by different $\psi_{h,t}$'s. The equivalence for the associated MRD codes are also not completely known. One of the reasons is that the comparing of coefficients of equations, such as \eqref{t:code_equiv}, becomes much more involved. We leave them as open questions.
	\end{remark}

\section*{Acknowledgments}
The research that led to the present paper was partially supported by a grant of the group GNSAGA of INdAM.
Yue Zhou is partially supported by Natural Science	Foundation of Hunan Province (No.\ 2019JJ30030) and Training Program for Excellent Young  Innovators of Changsha (No.\ kq1905052).

\end{document}